\documentclass[a4paper,reqno,intlimits]{amsart}
\usepackage{hyperref}
\usepackage{fullpage}
\usepackage{amsfonts}
\usepackage{amsmath}
\usepackage{amssymb}
\usepackage{amsthm}
\usepackage{latexsym} 
\usepackage[utf8]{inputenc}

\usepackage{tikz}
\usetikzlibrary{cd}

\usepackage{mathtools}

\def\N{\mathbb N}
\def\R{\mathbb R}
\def\C{\mathbb C}

\def\D{\mathcal D}
\def\CC{\mathcal C}\def\O{\mathcal O}

\def\E{\mathcal E}

\def\G{\mathcal G}

\def\U{\mathcal U}
\def\O{\mathcal O}
\def\M{\mathcal M}

\def\eps{\varepsilon}

\DeclareMathOperator{\id}{id}

\DeclareMathOperator{\WF}{WF}
\DeclareMathOperator{\real}{Re}
\DeclareMathOperator{\imag}{Im}
\DeclareMathOperator{\supp}{supp}

\DeclareMathOperator{\Char}{Char}

\theoremstyle{plain}
\newtheorem{Thm}{Theorem}[section]
\newtheorem{Prop}[Thm]{Proposition}
\newtheorem{Lem}[Thm]{Lemma}
\newtheorem{Cor}[Thm]{Corollary}

\theoremstyle{definition}
\newtheorem{Def}[Thm]{Definition}
\newtheorem{Ex}[Thm]{Example}
\newtheorem{Rem}[Thm]{Remark}

\numberwithin{equation}{section}

\begin{document}
	\title{Geometric Microlocal Analysis in Denjoy-Carleman Classes}
	\author{Stefan F{\"u}rd{\"o}s}
	\email{furdos@math.muni.cz}
	\address{Department of Mathematics and Statistics, Masaryk University, Brno, Czech Republic}
\begin{abstract}
A systematic geometric theory for the ultradifferentiable (non-quasianalytic and quasianalytic) wavefront
set similar to the well-known theory in the classic smooth and analytic setting is developed. In particular
an analogue of Bony’s Theorem and the invariance of the ultradifferentiable wavefront set under
diffeomorphisms of the same regularity is proven using a Theorem of Dyn’kin about the almost-analytic
extension of ultradifferentiable functions. Furthermore we prove a microlocal elliptic regularity theorem
for operators defined on ultradifferentiable vector bundles. As an application we show that Holmgren’s
theorem and several generalizations hold for operators with quasianalytic coefficients.
\end{abstract}
\keywords{ultradifferentiable wavefront set, Denjoy-Carleman classes, Bony's theorem,
quasianalytic uniqueness theorems}
\subjclass[2010]{35A18; 26E10; 35A30; 35A02}
	\maketitle
\section{Introduction}
The aim of this work is to establish a geometric theory for the wavefront set in ultradifferentiable classes introduced by H\"ormander \cite{MR0294849}
analogous to the one for the classical wavefront set.
There are a number of 
recent works dealing with this question, see Adwan-Hoepfner \cite{MR2718658}, Hoepfner-Medrado \cite{Hoepfner-MedradoPreprint}. In this paper we present a unified
approach to the problem, which also allows us to treat quasianalytic classes,
which the methods introduced up to now were not able to cover.

Regarding questions of the regularity of solutions of PDEs the wavefront set is a crucial notion introduced by Sato \cite{zbMATH03331746} in the analytic category and 
by H\"ormander \cite{MR0513000} in the smooth case. 
Their refinement of the singular support simplifies for example the proof of the classical elliptic regularity theorem considerably.

One of the basic features of both the smooth and analytic wavefront sets is that they are invariant under smooth and real-analytic changes of coordinates, respectively.
Hence it is possible to define the smooth (or analytic) wavefront of a distribution given on a smooth (or analytic) manifold. 
This is mainly due to the fact the smooth resp.\ analytic wavefront set  can either be described by the Fourier transform (H\"ormander's approach), 
boundary values of almost analytic resp.\ holomorphic functions (Sato's definiton)
or by the FBI transform (due to Bros and Iagolnitzer \cite{MR0399494}). The proof of the equivalence of these descriptions in the analytic category is due to Bony \cite{MR0650834}.

Various other notions of wavefront sets associated to microlocalizable structures have since then been  introduced, e.g.\ for Sobolov spaces, c.f.\ e.g.\ \cite{Lerner2010}.
In this paper we 
are interested in ultradifferentiable classes, that is spaces of smooth functions which
include strictly all real analytic functions. The most well known example of such classes are the Gevrey classes, see e.g.\
Rodino \cite{MR1249275}.

Generally spaces of ultradifferentiable functions are  defined by putting growth conditions either on the derivatives or the Fourier transform of its elements.
One family of ultradifferentiable classes, which includes the Gevrey classes,
is the category of Denjoy-Carleman classes. 
The elements of a Denjoy-Carleman class satisfy generalized Cauchy estimates of the form 
\begin{equation*}
\bigl\lvert \partial^\alpha f(x)\bigr\rvert\leq C h^{\lvert\alpha\rvert}M_{\lvert\alpha\rvert} 
\end{equation*}
on compact sets, where $C$ and $h$ are constants independent of $\alpha$ and $\M=(M_j)_j$ is a sequence of positive real numbers, 
the 
weight sequence associated to the Denjoy-Carleman class.
Such classes of smooth functions were first investigated by Borel and Hadamard, 
but were named after Denjoy and Carleman who characterized independently the  quasianalyticity of such a class using its weight sequence, 
c.f.\ the survey \cite{MR2384272} by Thilliez.

There is a rich literature concerning the Denjoy-Carleman classes and their properties.
It turns out that conditions on the weight sequence translate to stability conditions of the associated class.
For example, if $\M$ is a regular weight sequence in the sense of Dyn'kin
 \cite{zbMATH03751341}, 
then it is known that the Denjoy-Carleman class is closed under composition, solving ordinary differential equations and that the implicit function theorem holds in the class, 
c.f.\ e.g.\ Bierstone-Milman \cite{MR2061220}. 
Hence it makes sense in this situation to consider manifolds of Denjoy-Carleman type.
In section \ref{sec:DC} we summarize the facts about weight sequences and
 Denjoy-Carleman classes that are needed in the remainder of the paper.
We note also that it is possible to generalize Nagano's theorem \cite{MR0199865} to orbits of quasianalytic vector fields.

There have been several attempts to define wavefront sets with respect to Denjoy-Carleman classes, see e.g.\ Komatsu \cite{MR1178558} and Chung-Kim \cite{MR1492944}. 
But the most widereaching definition of an ultradifferentiable wavefront set both with respect to the conditions imposed on the weight sequence and scope of achieved results 
was given by H\"ormander \cite{MR0294849}.
Due to the relatively weak conditions that he imposed on the weight sequence 
H\"ormander was only able to define the 
ultradifferentiable wavefront set $\WF_\M u$ of distributions $u$ on real-analytic manifolds but not distributions defined on general
ultradifferentiable manifolds. H\"ormander's results are reviewed in section \ref{SecWF}.

The main result we need in order to proceed is a theorem of Dyn'kin \cite{zbMATH03751341}. 
He showed that for regular weight sequences each function in a regular Denjoy-Carleman 
class has an almost-analytic extension, whose $\bar{\partial}$-derivative satisfies
 near $\imag z=0$ a certain exponential decrease in terms of the weight sequence. 
We apply this result and several statements of H\"ormander \cite{MR1996773} in section \ref{sec:BV} to prove that 
the Denjoy-Carleman wavefront set can be characterized by such $\M$-almost-analytic extensions. 
Using this characterization it is possible to modify H\"ormander's proof 
of the  invariance of the wavefront set 
in the real-analytic case to show that in our situation
the ultradifferentiable wavefront set for distributions on Denjoy-Carleman manifolds can be well defined. 

In section \ref{sec:Bony} we show that $\WF_{\M}u$ can be characterized by the generalized FBI transform introduced by Berhanu and Hounie \cite{MR2864805}.
This generalizes results of Berhanu-Hailu\cite{MR3701268} and
Hoepfner-Medrado \cite{Hoepfner-MedradoPreprint}, 
in particular to quasianalytic classes.

As mentioned in the beginning, one of the fundamental results regarding the classical wavefront set is the elliptic regularity theorem 
which states in its microlocal form that we have for all partial differential operators $P$ with smooth coefficients that 
$\WF u\subseteq\WF Pu\cup \Char P$, where $\Char P$ is the set of characteristic points of $P$, for all distributions $u$.
Similarly H\"ormander proved that $\WF_\M u\subseteq \WF_\M u \cup \Char P$ holds for operators with real-analytic coefficients.
However, recently several authors, e.g.\ Albanese-Jornet-Oliaro\cite{MR2595651} and Pilipovi{\'c}-Teofanov-Tomi{\'c} \cite{Pilipovic:2016oj}, 
have used the pattern of H\"ormander's proof 
to show this inclusion for ultradifferentiable wavefront sets and operators with ultradifferentiable coefficients for variously defined ultradifferentiable classes.

Arguing similarly we prove in section \ref{sec:elliptic} that, if $\M$ is a regular weight sequence that satisfies an additional condition,
which is usually referred to in the literature as \emph{moderate growth}, see e.g.\ Thilliez \cite{MR2384272}, 
then $\WF_\M u\subseteq\WF_\M Pu\cup\Char P$ 
for operators $P$ with coefficients in the Denjoy-Carleman class associated to $\M$.
In fact, we show this inclusion for  operators with ultradifferentiable coefficients acting on distributional sections of ultradifferentiable vector bundles.

Following the approach given separately by Kawai \cite{MR0420735} and H\"ormander \cite{MR0294849} in the analytic case
we use the elliptic regularity theorem in section \ref{sec:uniq} to prove a generalization of Holmgren's Uniqueness Theorem 
to operators with coefficients in quasianalytic Denjoy-Carleman classes. 
Finally we give quasianalytic versions of the generalizations of the analytic Holmgren's theorem due to 
Bony \cite{MR0474426}, H\"ormander \cite{MR1275197}, Sj\"ostrand \cite{MR699623} and Zachmanoglou \cite{MR0299925}.

The author was supported by the Austrian Science Fund FWF, international cooperation  project Nr.\ I01776 and by the Czech Science Foundation GACR grant 17-19437S.
\section{Denjoy-Carleman classes}\label{sec:DC}
Troughout this article $\Omega$ is going to denote an open subset of $\R^n$.
A \emph{weight sequence} is a sequence of positive real numbers $(M_j)_{j\in\N_0}$ 
with the following properties
\begin{align*}
M_0&=1\\
M_j^2&\leq M_{j-1}M_{j+1}\qquad j\in\N.
\end{align*}
\begin{Def}
Let $\M=(M_j)_j$ be a weight sequence.
We say that a smooth function $f\in\E(\Omega)$ is \emph{ultradifferentiable of class $\{\M\}$} iff for every compact set 
$K\subset\subset \Omega$ there exist  constants $C$ and $h$ such that for all multi-indices $\alpha\in\N_0^n$
\begin{equation}\label{DC-classDef}
\bigl\lvert D^{\alpha}f(x)\bigr\rvert\leq Ch^{\lvert\alpha\rvert}M_{\lvert\alpha\rvert} \qquad x\in K.
\end{equation}

We denote the space of ultradifferentiable functions of class $\{M\}$ on $\Omega$ as $\E_\M(\Omega)$.
Note that $\E_\M(\Omega)$ is always a subalgebra of $\E(\Omega)$ (Komatsu \cite{MR0320743}).
\end{Def}
\begin{Ex}
For any $s\geq 0$ consider the sequence $\M^s=((k!)^{s+1})_k$. The space of ultradifferentiable functions 
associated to $\M^s$ is the well known space of Gevrey functions $\mathcal{G}^{s+1}=\E_{\M^s}$ of order $s+1$, c.f.\ e.g.\ Rodino \cite{MR1249275}.
If $s\!=\!0$ then $\G^1\!=\!\E_{\M^0}\!=\!\mathcal{O}$ is the space of real-analytic functions.
\end{Ex}
\begin{Rem}
It is easy to see that $\E_\M(\Omega)$ is an infinite-dimensional vector space, 
since it contains all polynomials.
In fact $\E_\M(\Omega)$ is a complete locally convex space, see e.g.\ Komatsu \cite{MR0320743}.
The topology on $\E_\M(\Omega)$ is defined as follows. If $K\subset\subset \Omega$ is a compact set 
such that $K=\overline{K^\circ}$ 
then we define for $f\in\E(K)$
\begin{equation*}
\lVert f\rVert_K^h:=\sup_{\substack{x\in K\\ \alpha\in\N^n_0}}\biggl\lvert\frac{D^{\alpha}f(x)}{h^{\lvert\alpha\rvert}M_{\lvert\alpha\rvert}}\biggr\rvert
\end{equation*}
and set
\begin{equation*}
\E^h_\M(K):=\bigl\{f\in\E(K)\mid \lVert f\rVert_K^h<\infty\bigr\}.
\end{equation*}

It is easy to see that $\E^h_\M(K)$ is a Banach space. Moreover, $\E^h_\M(K)\subsetneq\E^k_\M(K)$ 
for $h<k$ and the inclusion mapping $\iota^k_h:\E^h_\M(K)\rightarrow\E^k_\M(K)$ is compact.
Hence the space
\begin{equation*}
\E_\M(K):=\bigl\{f\in\E(K)\mid \exists h>0\colon\lVert f\rVert_K^h<\infty\bigr\}=\varinjlim_h\, \E^h_\M(K)
\end{equation*}
is a (LB)-space. We can now write 
\begin{equation*}
\E_{\M}(\Omega)=\varprojlim_K\,\E_{\M}(K)
\end{equation*}
as a projective limit. For more details on the topological structure of $\E_\M(\Omega)$ see Komatsu \cite{MR0320743}.
\end{Rem}
We call $\E_\M(\Omega)$ also the \emph{Denjoy-Carleman class on $\Omega$ associated
to the weight sequence $\M$}.

If $\M$ and $\mathcal{N}$ are two weight sequences then
\begin{equation*}
\M\preccurlyeq\mathcal{N}\,:\Longleftrightarrow\, \sup_{k\in\N_0}\biggl(\frac{M_k}{N_k}\biggr)^{\tfrac{1}{k}}<\infty
\end{equation*}
defines a reflexive and transitive relation on the space of weight sequences.
Furthermore it induces an equivalence relation by setting
\begin{equation*}
\M\approx\mathcal{N}\,:\Longleftrightarrow\, \M\preccurlyeq\mathcal{N}\;\text{and}\; \mathcal{N}\preccurlyeq\M.
\end{equation*}
It holds that $\E_\M\subseteq\E_\mathcal{N}$ if and only if $\M\preccurlyeq\mathcal{N}$
and $\E_\M=\E_{\mathcal{N}}$ if and only if $\M \approx\mathcal{N}$, see Mandelbrojt \cite{MR0051893}, 
c.f.\ also \cite{MR3285413} and \cite{MR2384272}. 
For example, if $r<s$ then $\mathcal{G}^{r+1}\subsetneq\mathcal{G}^{s+1}$.

The weight function $\omega_\M$ (c.f.\ Mandelbrojt \cite{MR0051893} and Komatsu \cite{MR0320743}) associated to the weight sequence $\M$ is defined by
\begin{align*}
\omega_\M(t)&:=\sup_{j\in\N_0}\log\frac{t^j}{M_j}  \qquad t>0,\\
\omega_\M(0)&:=0.
\end{align*}
We note that $\omega_\M$ is a continuous increasing function on $[0,\infty)$, vanishes on the interval $[0,1]$
 and $\omega_\M\circ\exp$ is convex.
 In particular $\omega_\M(t)$ increases faster than $\log t^p$ for any $p>0$ as $t$ tends to infinity  
It is possible to extract the weight sequence from the weight function , i.e.
\begin{equation*}
M_k=\sup_t\frac{t^k}{e^{\omega_{\M}(t)}},
\end{equation*}
see \ Mandelbrojt \cite{MR0051893} or Komatsu \cite{MR0320743}.

If $f$ and $g$ are two continuous functions defined on $[0,\infty)$ then we write $f\sim g$ iff 
$f(t)=O(g(t))$ and $g(t)=O(f(t))$ for $t\rightarrow \infty$.
It can be shown that the weight function $\omega_{s}$ for the Gevrey space $\mathcal{G}^{s+1}$ satisfies 
\begin{equation*} 
\omega_s(t)\sim t^{\tfrac{1}{s+1}}.
\end{equation*}

Sometimes the classes $\E_\M$ are defined using the sequence $m_k=\frac{M_k}{k!}$ instead of $(M_k)_k$ and 
\eqref{DC-classDef} is replaced by
\begin{equation*}
\bigl\lvert D^{\alpha}f(x)\bigr\rvert\leq C h^{\lvert\alpha\rvert}\lvert\alpha\rvert!m_{\lvert\alpha\rvert}.
\end{equation*}
Infrequently the sequences $\mu_k=\tfrac{M_{k+1}}{M_k}$ or $L_k=M_k^{\tfrac{1}{k}}$ are also used, with an accordingly modified version of \eqref{DC-classDef},
c.f.\ also Remark \ref{HoeClassCom}.
The main reason for the different ways of defining the Denjoy-Carleman classes is the following.
In order to show that these classes satisfy certain properties, 
like the inverse function theorem, one has to put certain conditions 
on the defining data of  the spaces, i.e.\ the weight sequence, c.f.\ e.g.\ Rainer-Schindl \cite{MR3462072}. 
Often these conditions are easier to write down in terms of these other sequences instead of using $(M_j)_j$.
In the following our point of view is that  the sequences 
$(M_k)_k$, $(m_k)_k$, $(\mu_k)_k$ and $(L_k)_k$ are all associated to the
weight sequence $\M$.
We are going to use especially the two sequences $(m_j)_j$ and $(M_j)_j$ 
indiscriminately.

We may note that sometimes ultradifferentiable functions associated to the weight sequence $\M$ are defined 
as smooth functions satisfying \eqref{DC-classDef} for all $h>0$ on each compact $K$, see e.g.\  Ehrenpreis \cite{MR0285849}.
One says then that $f$ is ultradifferentiable of class $(\M)$ 
and the corresponding space is the Beurling class associated to $\M$. 
On the other hand $\E_\M$ is then usually called the Roumieu class associated to $\M$, 
c.f.\ e.g.\ Komatsu \cite{MR0320743} or Rainer-Schindl \cite{MR3462072}.

From now on we shall put certain conditions on the weight sequences under consideration.
\begin{Def}\label{Link}
We say that a weight sequence $\M$ is \emph{regular} iff it satisfies the following conditions:
\begin{align*}
\tag{\text{M}1} &m_0=m_1=1\label{normalization}\\
\tag{\text{M}2} \label{derivclosed} &\sup_k\sqrt[\leftroot{1}k]{\frac{m_{k+1}}{m_k}}<\infty\\
\tag{\text{M}3}\label{stlogconvex} &m_k^2\leq m_{k-1}m_{k+1}\qquad k\in\N\\
\tag{\text{M}4}\label{analyticincl} &\lim_{k\rightarrow\infty}\sqrt[\leftroot{2}k]{m_k}=\infty
\end{align*}
\end{Def}
The last condition just means that the space $\O$ of all real-analytic functions is strictly contained in 
$\E_\M$ whereas the first is an useful normalization condition that will help simplify certain computations.
It is obvious that if we replace in \eqref{normalization} the number 
$1$ with some other positive real number we would not
change the resulting space $\E_\M$.

If $\M$ is a regular weight sequence, then it is well known that the associated Denjoy-Carleman class satisfies
certain stability properties, c.f.\ e.g.\ Bierstone-Milman \cite{MR2061220} or Rainer-Schindl \cite{MR3462072}.
For example $\E_\M$ is \emph{closed under differentiation}, i.e.\ if $f\in\E_\M(\Omega)$ then 
$D^\alpha f\in\E_\M(\Omega)$ for all $\alpha\in\N^n_0$.
\begin{Rem}\label{HadamardLemma}
The fact that $\E_\M(\Omega)$ is closed under differentiation 
 implies immediately another stability condition, 
namely \emph{closedness under division by a coordinate} (c.f.\ Bierstone-Milman \cite{MR2061220}):

Suppose that $f\in\E_\M(\Omega)$ and $f(x_1,\dotsc ,x_{j-1},a,x_{j+1},\dotsc ,x_n)=0$ 
for some fixed $a\in\R$ and
all $x_k$, $k\neq j$, with the property that $(x_1,\dotsc ,x_{j-1},a,x_{j+1},\dotsc ,x_n)\in\Omega$. 
Then we apply the Fundamental Theorem of Calculus to the function 
\begin{equation*}
f_j: t\longmapsto f(x_1,\dotsc,x_{j-1},t(x_j-a)+a,x_{j+1},\dotsc, x_n)
\end{equation*}
and obtain 
\begin{equation*}
f(x)=\int_{0}^1\! \frac{\partial f_j}{\partial t}(t)\, dt =(x_j-a)\int_{0}^1\!
\frac{\partial f}{\partial x_j}(x_1,\dotsc,x_{j-1},t(x_j-a)+a,x_{j+1},\dotsc, x_n)\,dt=(x_j -a) g(x).
\end{equation*}
It is easy to see that $g\in\E_\M(\Omega)$ using $\tfrac{\partial f}{\partial x_j}\in\E_\M(\Omega)$.
\end{Rem}

For the proof of the properties above only \eqref{derivclosed} was used. 
If we apply also \eqref{stlogconvex} then it is possible to show that $\E_\M(\Omega)$ is \emph{inverse closed},
i.e.\ if $f\in\E_\M(\Omega)$ does not vanish at any point of $\Omega$ then
\begin{equation*}
\frac{1}{f}\in\E_\M(\Omega),
\end{equation*}
c.f.\ Rainer-Schindl \cite{MR3462072} and the remarks therein.

In fact, if $\M$ is a regular weight sequence then the associated Denjoy-Carleman class satisfies also
the following stability properties.
\begin{Thm}\label{CMStability}
Let $\M$ be a regular weight sequence and $\Omega_1\subseteq\R^m$ and 
$\Omega_2\subseteq\R^n$ open sets. Then the following holds:
\begin{enumerate}
\item The class $\E_\M $ is \emph{closed under composition} (Roumieu \cite{MR0158261} see also Bierstone-Milman\cite{MR2061220}) 
i.e.\ let $F\!:\Omega_1\rightarrow\Omega_2$ be an $\E_\M$-mapping, that is each
component $F_j$ of $F$ is ultradifferentiable of class $\{\M\}$ in $\Omega_1$, and $g\in\E_\M(\Omega_2)$.
Then also $g\circ F\in \E_\M(\Omega_1)$.
\item The \emph{inverse function theorem} holds in the Denjoy-Carleman class $\E_\M$ (Komatsu \cite{MR531445}): 
Let $F:\Omega_1\rightarrow \Omega_2$ be an $\E_\M$-mapping and $p_0\in\Omega_1$ such that 
the Jacobian $F^\prime(p_0)$ is invertible. Then there exist neighbourhoods $U$ of $p_0$ in $\Omega_1$
and $V$ of $q_0=F(x_0)$ in $\Omega_2$ and a $\E_\M$-mapping $G:V\rightarrow U$ such that $G(q_0)=p_0$
and $F\circ G=\id_V$.
\item The \emph{implicit function theorem} is valid in $\E_\M$ (Komatsu \cite{MR531445}): Let $F:\R^{n+d}\supseteq\Omega\rightarrow \R^d$
be a $\E_\M$-mapping and $(x_0,y_0)\in \Omega$ such that $F(x_0,y_0)=0$ and
$\tfrac{\partial F}{\partial y}(x_0,y_0)$
is invertible. Then there exist open sets $U\!\subseteq\!\R^n$ and $V\!\subseteq\! \R^d$ with 
$(x_0,y_0)\!\in\! U\!\times\! V\!\subseteq\!\Omega$ and an $\E_\M$-mapping $G:\,U\rightarrow V$ 
such that $G(x_0)=y_0$
and $F(x,G(x))=0$ for all $x\in V$.
\end{enumerate}
\end{Thm}
Furthermore it is true that
 $\E_\M(\Omega)$ is \emph{closed under solving ODEs}, to be more specific the following theorem holds.
\begin{Thm}[Yamanaka \cite{MR1128962} see also Komatsu \cite{MR575993}]\label{ClosednessODE}
Let $\M$ be a regular weight sequence, $0\in I\subseteq\R$ an open interval, 
$U\subseteq \R^n$, $V\subseteq\R^d$ open and $F\in\E_\M(I\times U\times V)$. 

Then the initial value problem
\begin{align*}
x^\prime(t)&=F(t,x(t),\lambda) & t&\in I,\,\lambda \in V\\
x(0)&=x_0 &x_0 &\in U
\end{align*}
has locally a unique solution $x$ that is ultradifferentiable near $0$.

More precisely, there is an open set $\Omega\subseteq I\times U\times V$ that contains the point 
$(0,x_0,\lambda)$ and an $\E_\M$-mapping $x=x(t, y,\lambda): \Omega\rightarrow U$ such that
the function $t\mapsto x(t,y_0,\lambda_0)$ is the solution of the initial value problem
\begin{align*}
x^\prime (t)&=F(t,x(t),\lambda_0)\\
x(0)&=y_0.
\end{align*}
\end{Thm}

For any regular weight sequence $\M$ we can define the associated weight by
\begin{equation}\label{Def:weight}
h_\M (t)=\inf_k t^km_k \quad \text{if } t>0\quad \text{\&}\quad h_\M(0)=0.
\end{equation}
Similarly to above we have that
\begin{equation*}
m_k=\sup_t\frac{h_\M(t)}{t^k}
\end{equation*}
In order to describe the connection between the weight and the weight function associated to a regular weight sequence we set
\begin{align*}
\tilde{\omega}_\M(t)&:=\sup_{j\in\N_0}\log\frac{t^j}{m_j}\\
\tilde{h}_\M (t)&=\inf_k t^kM_k 
\end{align*}
for $t>0$ and $\tilde{\omega}_\M(0)=\tilde{h}_\M(0)=0$.
\begin{Lem}\label{lem-weight-con}
If $\M$ is a regular weight sequence then
\begin{equation}\label{con}
\begin{split}
h_\M(t)&=e^{-\tilde{\omega}_\M\bigl(\tfrac{1}{t}\bigr)}\\
\tilde{h}_\M(t)&=e^{-\omega_\M\bigl(\tfrac{1}{t}\bigr)}
\end{split}
\end{equation}
\end{Lem}
\begin{proof}
We prove only the equality for $h_\M$. Of course, the verification of the other equation is completely analogous.
If $t>0$ is chosen arbitrarily we have by the monotonicity of the exponential function that 
\begin{equation*}
\exp \biggl(\tilde{\omega}_\M\biggl(\frac{1}{t}\biggr)\biggr)=\exp\biggl(\sup_k\log \frac{1}{m_kt^k}\biggr)=
\sup_k\frac{1}{m_kt^k}=\frac{1}{\inf_k m_k t^k}=\frac{1}{h_\M (t)}.
\end{equation*}
\end{proof}
We obtain that $h_\M$ is continuous with values in $[0,1]$, equals $1$ on $[1,\infty)$
 and goes more rapidly to $0$ than $t^p$ for any $p>0$ for $t\rightarrow 0$.
  Albeit the weight function is the prevalant concept,
the weight was used e.g.\ by Dyn'kin \cite{MR0587795,zbMATH03751341} and Thilliez \cite{MR2011916}.
\begin{Ex}
If $\M=\M^s$ is the Gevrey sequence of order $s$ then we know already that
the associated weight function satisfies
$\omega_s(t)\sim t^{\frac{1}{1+s}}$. 
Hence $\eqref{con}$ shows for $s>0$ that if we set
\begin{equation*}
f_s(t)=e^{-\tfrac{1}{t^s}}
\end{equation*}
then there are constants $C_1$, $C_2$, $Q_1$ and $Q_2>0$ such that
\begin{equation*}
C_1 f_s\bigl(Q_1t\bigr)\leq h_s(t)\leq C_2 f_s\bigl(Q_2t\bigr) 
\end{equation*}
for $t>0$.
\end{Ex}

It is well known (see e.g.\ Mather \cite{MR0412461} or Melin-Sj{\"o}strand \cite{MR0431289} 
) that a function $f$ is smooth on $\Omega$ 
if and only if there is an 
almost-analytic extension $F$ of $f$, i.e.\ there exists a smooth function $F$ on some open set 
$\tilde{\Omega}\subseteq\C^n$ with $\tilde{\Omega}\cap\R^n=\Omega$ such that 
\begin{equation*}
\bar{\partial}_jF=\frac{\partial}{\partial \bar{z}_j}F=\frac{1}{2}\biggl(\frac{\partial}{\partial x_j}+i\frac{\partial}{\partial y_j}\biggr)F
\end{equation*}
is flat on $\Omega$ and $F|_\Omega=f$. 
The idea is now that if $f$ is ultradifferentiable 
then one should find an extension $F$ of $f$ such that the regularity of $f$ is 
translated in a certain uniform decrease of $\tilde{\partial}_jF$ near $\Omega$ (c.f.\ Dyn'kin \cite{MR1253229}). 
Such extensions were constructed e.g.\ by Petzsche-Vogt \cite{MR737333} and Adwan-Hoepfner \cite{MR2718658}
under relative restrictive conditions on the weight sequence. 
The most general result in this regard though was given by Dyn'kin \cite{MR0587795} (c.f.\ the english translation in \cite{zbMATH03751341}).
\begin{Thm}\label{Dynkin1}
Let $\M$ be a regular weight sequence, $K\subset\subset\R^n$ a compact convex set with 
$K=\overline{K^\circ}$.
Then $f\in\E_\M(K)$ if and only  if there exists a test function $F\in\D(\C^n)$ with $F|_K=f$ and 
if there are constants $C,Q>0$ such that
\begin{equation}
\bigl\lvert\bar{\partial}_jF(z)\bigr\rvert\leq C h_\M(Qd_K(z))
\end{equation}
where $1\leq j\leq n$ and $d_K$ is the distance function with respect to $K$ on $\C^n\!\setminus\!K$.
\end{Thm}
We shall note that Dyn'kin used the function $h_1(t)=\inf_{k\in\N} m_k t^{k-1}$ instead of the weight $h_\M$\footnote{$h_1$ is in fact the weight associated to the shifted sequence $(m_{k+1})_{k}$}.
But we observe that 
\begin{equation*}
h_\M(t)=\inf_{k\in\N_0}m_kt^k\leq t\inf_{k\in\N}m_k t^{k-1}=th_1(t)\leq Ct\inf_{k\in\N}m_{k-1}t^{k-1}=Cth_{\M}(t),
\end{equation*}
where we used \eqref{derivclosed}.
Since $h_\M$ is rapidly decreasing for $t\rightarrow 0$ we can 
interchange these two functions in the formulation of Theorem \ref{Dynkin1}. 
In fact, Dyn'kin's proof gives immediately the following result.
\begin{Cor}\label{CharMalmostanalytic}
Let $\M$ be a regular weight sequence, $p\in\Omega$ and $f\in\D^\prime(\Omega)$. 
If $f$ is ultradifferentiable of class $\{\M\}$ near $p$, 
i.e.\ there exists a compact neighbourhood $K$ of $p$ such that $f\vert_K\in\E_\M(K)$,
then  there are an open neighbourhood $W\subseteq\Omega$,
a constant $\rho>0$ and a function $F\in\E(W+iB(0,\rho))$ such that $F|_W=f|_W$ and
\begin{equation}\label{Malmostest}
\bigl\lvert\bar{\partial}_j F(x+iy)\bigr\rvert\leq Ch_\M (Q\lvert y\rvert)
\end{equation}
for some positive constants $C,Q$ and all $1\leq j\leq n$ and $x+iy\in W+iB(0,\rho)$.
\end{Cor}

One of the main questions in the study of ultradifferentiable functions is 
if the class under consideration behaves more like the ring of real-analytic functions
or the ring of smooth functions. E.g., does the class contain flat functions, 
that means nonzero elements whose Taylor series at some point vanishes?
That leads to following definition.
\begin{Def}
Let $E\subseteq\E(\Omega)$ be a subalgebra. We say that $E$ is quasianalytic iff 
for $f\in E$ the fact that $D^\alpha f(p)=0$ for some $p\in\Omega$ and all $\alpha\in\N_0^n$ implies
that $f\equiv 0$ in the connected component of $\Omega$ that contains $p$.
\end{Def}
In the case of Denjoy-Carleman classes quasianalyticity is characterized by the following theorem.
\begin{Thm}[Denjoy\cite{zbMATH02601219}-Carleman\cite{zbMATH02599917,zbMATH02598188}]
The space $\E_\M(\Omega)$ is quasianalytic
if and only if
\begin{equation}\label{quasiCond}
\sum_{k=1}^\infty\frac{M_{k-1}}{M_k}=\infty.
\end{equation}
\end{Thm}
We say that a weight sequence is quasianalytic iff it satisfies \eqref{quasiCond}
and non-quasianalytic otherwise.
\begin{Ex}\label{QuasiEx}
Let $\sigma>0$ be a parameter. We define a family $\mathcal{N}^\sigma$ 
of regular weight sequences by $N_0=N_1=1$ and
\begin{equation*}
N_k^\sigma=k!\bigl(\log(k+e )\bigr)^{\sigma k}
\end{equation*}
for $k\geq 2$.
The weight sequence $\mathcal{N}^\sigma$ is quasianalytic if and only if 
$0<\sigma\leq 1$, see Thilliez \cite{MR2384272}.
\end{Ex}
\begin{Rem}
Obviously $\D_\M(\Omega)=\D(\Omega)\cap\E_\M(\Omega)$ is nontrivial if and only if
$\E_\M(\Omega)$ is non-quasianalytic, c.f.\ e.g.\ Rudin \cite{MR924157}.
It is well known that the sequences $\M^s$ are non-quasianalytic if and only if $s>0$.
In fact there is a non-quasianalytic regular weight sequence $\tilde{\M}$ such that 
$\tilde{\M}\precnapprox\M^s$ for all $s>0$, see Rainer-Schindl \cite[p.125]{MR3285413}.
Hence 
\begin{equation*}
\mathcal{O}\subsetneq\E_{\tilde{\M}}\subsetneq\bigcap_{s>0}\mathcal{G}^{s+1}.
\end{equation*}
\end{Rem}
Using Theorem \ref{CMStability} we are able to define
\begin{Def}
Let $M$ be a smooth manifold and $\M$ a regular weight sequence. We say that $M$ is an ultradifferentiable 
manifold of class $\{\M\}$ iff there is an atlas $\mathcal{A}$ of $M$ that consists of charts such that 
\begin{equation*}
\varphi^\prime\circ\varphi^{-1}\in\E_{\M}
\end{equation*}
for all $\varphi,\varphi^{\prime}\in\mathcal{A}$.
\end{Def}

A mapping $F\!:\, M\rightarrow N$ between two manifolds of class $\{\M \}$ is 
ultradifferentiable of class $\{\M \}$ iff $\psi\circ F\circ \varphi^{-1}\in\E_{\M}$
for any charts $\varphi$ and $\psi$ of $M$ and $N$, respectively.
We can now consider the category of ultradifferentiable manifolds of class $\{\M\}$. 
We denote by $\mathfrak{X}_\M(M)=\E_\M(M, TM)$ the Lie algebra of ultradifferentiable vector fields on $M$.
Note that, if $\M$ is a regular weight sequence, an integral curve of an ultradifferentiable vector field of class $\{\M\}$ 
is an $\E_\M$-curve by Theorem \ref{ClosednessODE}.

 These considerations allow us to state 
a quasianalytic version of Nagano's theorem \cite{MR0199865}.
\begin{Thm}\label{NaganoThm}
Let  $U$ be an open neighbourhood of $p_0\in\R^n$ and $\M$ a quasianalytic regular weight sequence.
Furthermore let $\mathfrak{g}$ be a Lie subalgebra of $\mathfrak{X}_\M(U)$
 that is also an $\E_\M$-module, i.e.\ if $X\in\mathfrak{g}$ and $f\in\E_\M(U)$ then $fX\in\mathfrak{g}$.
 
 Then there exists an ultradifferentiable submanifold $W$ of class $\{\M\}$ in $U$, such that
 \begin{equation}\label{Naganoeq}
 T_pW=\mathfrak{g}(p)\qquad \forall p\in W.
 \end{equation}
 Moreover, the germ of $W$ at $p_0$ is uniquely defined by this property.
\end{Thm}
The proof of Theorem \ref{NaganoThm} is the same as in the analytic version, c.f.\ e.g.\ Baouendi-Ebenfelt-Rothschild \cite{MR1668103}.
We call the uniquely defined germ $\gamma_{p_0}(\mathfrak{g})$ of the manifold constructed in Theorem \ref{NaganoThm} the local Nagano leaf of $\mathfrak{g}$ at $p_0$.
From now on all Lie algebras of ultradifferentiable vector fields that are considered are assumed to be also $\E_\M$-modules.

Following Nagano \cite{MR0199865}, c.f.\ also Baouendi-Ebenfelt-Rothschild \cite{MR1668103}, we can also give a global version of Theorem \ref{NaganoThm}.
\begin{Thm}\label{globalNaganothm}
Let $\M$ be a quasianalytic regular weight sequence. If $\mathfrak{g}$ is a Lie subalgebra of
 $\mathfrak{X}_\M(\Omega)$ then $\mathfrak{g}$ admits a foliation of $\Omega$, 
 that is a partition of $\Omega$ by maximal integral manifolds.
\end{Thm}

Before we close this section we need to introduce another condition for weight sequences.
Let $\M$ be a weight sequence. We say that $\M$ is of \emph{moderate growth} iff
there are constants $C$ and $\rho$ such that
\begin{equation}\tag{$\text{M}2^\prime$}\label{mg}
M_{j+k}\leq C\rho^{j+k}M_jM_k
\end{equation}
for all $(j,k)\in\N_0^2$.
Both the Gevrey sequences $\M^s$ and the sequences $\mathcal{N}^\sigma$ from  Example 
\ref{QuasiEx} satisfy \eqref{mg} for all $s$ and $\sigma$, respectively.

For a discussion of this condition, see e.g.\ Komatsu \cite{MR0320743}. 
Here we only mention two facts. First, for any weight sequence $\M$, if \eqref{mg} holds then \eqref{derivclosed} is also satified.
Furthermore, if $\M$ satisfies \eqref{mg} then there is some $s>0$ such that $\E_\M\subseteq\mathcal{G}^{1+s}$, c.f.\
e.g.\ Thilliez \cite{MR2011916}.
On the other hand consider the regular weight sequence $\mathcal{L}$ given by 
$L_0=L_1=1$ and $L_k=k!2^{k^2}$ if $k\geq 2$. Then $\G^{1+s}\subseteq\E_{\mathcal{L}}$
for all $s\geq 0$.

\section{The ultradifferentiable wavefront set}\label{SecWF}
In this and the following two sections we always assume that $\M$ is a regular weight sequence.

In 1971 H{\"o}rmander \cite{MR0294849} proved the following local characterization of $\E_\M$ via the Fourier transform:
\begin{Prop}\label{MCharFT}
Let $u\in\D^\prime (\Omega)$ and $p_0\in\Omega$. 
Then $u$ is ultradifferentiable of class $\{\M\}$ near $p_0$ 
if and only if there are an open neighbourhood $V$ of $p_0$, 
a bounded sequence $(u_N)_N\subseteq\E^\prime (U)$
such that $u|_V=(u_N)|_V$ and some constant $Q>0$ so that
\begin{equation*}
\sup_{\substack{\xi\in\R^n\\ 
N\in\N_0}}\frac{\lvert\xi\rvert^N\lvert \hat{u}_N(\xi)\rvert}{Q^N M_N}<\infty.
\end{equation*}
\end{Prop}
Subsequently he used this fact to define analogously to the smooth category:
\begin{Def}\label{WF-M Def1}
Let $u\in\D^\prime (\Omega)$ and $(x_0,\xi_0)\in T^*\Omega\!\setminus\!\{ 0\}$.
We say that $u$ is \emph{microlocally ultradifferentiable of class $\{\M\}$} at $(x_0,\xi_0)$ 
iff there is a bounded sequence 
$(u_N)_N\subseteq\E^\prime (\Omega)$ such that $u_N\vert_V\equiv u\vert_V$, where $V\in\U(x_0)$
 and a conic neighbourhood $\Gamma$ of $\xi_0$  such that for some constant $Q>0$
\begin{equation}\label{WF-M Estimate1}
\sup_{\substack{\xi\in\Gamma\\
 N\in\N_0}} \frac{\lvert\xi\rvert^N\lvert\hat{u}_N\rvert}{Q^N M_N}<\infty.
\end{equation}
The ultradifferentiable wavefront set $\WF_\M u$ is then defined as
\begin{equation*}
\WF_\M u:=\bigl\{(x,\xi)\in T^*\Omega\!\setminus\!\{ 0\}\mid u\text{ is not microlocally
	ultradiff.\ of class }\{\M\}
\text{ at }(x,\xi)\bigr\}.
\end{equation*}
\end{Def}
\begin{Rem}\label{HoeClassCom}
We need to point out that H{\"o}rmander in \cite{MR0294849} defined $\WF_\M$ for weight sequences that 
satisfy weaker conditions then those we imposed in Definition \ref{Link}. 
He required, as we have done, \eqref{derivclosed} and that $\O\subseteq\E_{\M}$,
but \eqref{stlogconvex} is replaced by the monotonic growth of the sequence 
\begin{equation}\label{HoeClass}
L_N=(M_N)^{\frac{1}{N}}.
\end{equation}
This condition still implies that $\E_{\M}$ is an algebra but gives only that $\E_{\M}$ is closed 
under composition with analytic mappings. 

More precisely,  in terms of the sequence $(L_N)_N$ the conditions that H\"ormander imposed take the following form.
First, $N\leq L_N$ and $L_{N+1}\leq C L_N$ for all $N$ and a constant $C>0$ independent of $N$. 
Furthermore as mentioned before the sequence $(L_N)_N$  is also assumed to be increasing.
 
 Note that his classes might not even be defined by weight sequences in the sense
of section \ref{sec:DC}.
Hence H{\"o}rmander in \cite{MR1996773} was able to define 
$\WF_{\M}u$ for distributions $u$ on real analytic manifolds 
but not on arbitrary ultradifferentiable manifolds of class $\{\M\}$;
note that the implicit function theorem may not hold in an arbitrary ultradifferentiable class defined by
weight sequences obeying his conditions.
Similarly he proved that 
\begin{equation*}
\WF_\M u\subseteq \WF_\M Pu\cup \Char P
\end{equation*}
 for linear partial differential operators $P$ with analytic coefficients but not for operators whose coefficients might be only of class $\{\M\}$.

As mentioned before it is possible to modify the arguments of H{\"o}rmander in the case of regular weight sequences to show that
the above inclusion holds for partial differential operators with ultradifferentiable coefficients as long as $\M$ is regular and of moderate growth. 
Similarly we are able to define $\WF_\M u$ for distributions defined on manifolds of class $\{\M\}$ (for regular $\M$), 
in this instance using 
Dyn'kin's almost-analytic extension of ultradifferentiable functions 
(i.e.\ Corollary \ref{CharMalmostanalytic}).

However, since regular weight sequences also fulfill the conditions of 
H{\"o}rmander we can use all of his results on $\WF_{\M}$.
Indeed, in terms of $L_N$, we have that \eqref{analyticincl} implies that $k\leq \gamma L_k$ for all $k\in\N_0$ and a constant $\gamma>0$ independent of $k$ 
by Sterling's formula whereas
\eqref{derivclosed} is equivalent to the existence of a constant $A>0$ such that $L_k\leq A L_{k-1}$.
We note that the last estimate implies
$L_N\leq A^N$
for $N\in\N_0$ since $L_1=1$. On the other hand,
 it is well known that if $(M_N)_N$ satisfies \eqref{stlogconvex} 
then $(L_N)_N$ is an increasing sequence, see e.g.\ Mandelbrojt \cite{MR0051893}.
\end{Rem}
The following result by H{\"o}rmander shows that we may choose the distributions $u_N$ 
in Definition \ref{WF-M Def1} in a special manner.
\begin{Prop}[\cite{MR1996773} Lemma 8.4.4.]\label{WF-M Charakterisierung}
Let $u\in\D^\prime(\Omega)$ and let $K\subset\Omega$ be compact, $F\subseteq\R^n$ a closed cone 
such that $\WF_\M u\cap (K\times F)=\emptyset$. If $\chi_N\in\D (K)$ and for all $\alpha$
\begin{equation*}
\bigl\lvert D^{\alpha +\beta}\chi_N\rvert\leq C_{\alpha} h_{\alpha}^{\lvert\beta\rvert}M_N^{\tfrac{\lvert\beta
\rvert}{N}} \qquad \lvert\beta\rvert\leq N
\end{equation*}
for some constants $C_\alpha$, $h_{\alpha}>0$ 
then it follows that $\chi_N u$ is bounded in $\E^{\prime S}$ if $u$ is of order $S$ in a neighbourhood 
of $K$, and further
\begin{equation*}
\lvert\widehat{\chi_N u}(\xi)\rvert\leq C \frac{Q^N M_N}{\lvert\xi\rvert^N} \qquad N\in\N,\; \xi\in F
\end{equation*}
for some constants $C,Q>0$.
\end{Prop}
We summarize the basic properties of $\WF_\M$ according to H{\"o}rmander \cite{MR1996773}.
\begin{Thm}[\cite{MR1996773} Theorem 8.4.5-8.4.7]\label{WF-MProperties}
Let $u\in\D^{\prime}(\Omega)$ and $\M$, $\mathcal{N}$ be two weight sequences. Then we have 
\begin{enumerate}
\item $\WF_\M u$ is a closed conic subset of $\Omega\times\R^n\!\setminus\!\{0\}$.
\item The projection of $\WF_\M u$ in $\Omega$ is
\begin{equation*}
\pi_1\bigl(\WF_\M u\bigr)=\mathrm{sing}\,\supp_{\M} u=
\overline{\bigl\{x\in\Omega \;\vert\; \nexists V\in\U(x):\, u\vert_V\in\E_\M (V)\}}
\end{equation*}
\item $\WF u\subseteq\WF_{\mathcal{N}}u\subseteq \WF_{\M}u$ if $\M\preccurlyeq \mathcal{N}$.
\item If $P=\sum p_{\alpha}D^{\alpha}$ is a partial differential operator with ultradifferentiable coefficents
of class $\{\M\}$ then $\WF_{\M} Pu\subseteq \WF_\M u$.
\end{enumerate}
\end{Thm}
Additionally we note that $\WF_\M u$ satisfies the following \emph{microlocal reflection property}:
\begin{equation}\label{microreflprop}
(x,\xi)\notin\WF_\M u \Longleftrightarrow (x,-\xi)\notin\WF_\M \bar{u}
\end{equation}
In particular, if $u$ is a real-valued distribution, i.e.\ $\bar{u}=u$, then 
$\WF_\M u\vert_{x}:=\{\xi\in\R^n\mid (x,\xi)\in\WF_\M u\}$ is symmetric at the origin.
\begin{Ex}
It is easy to see that $\WF_\M \delta_p=\{p\}\times\R^n\!\setminus\!\{0\}$ 
for any regular weight sequence $\M$.
\end{Ex}
\begin{Rem}\label{Discussion}
The complicated form of Definition \ref{WF-M Def1} compared with the definition of the smooth wavefront set 
stems from the fact that quasianalytic weight sequences are allowed. 
Thus in general there may not be any nontrivial test functions of class $\{\M\}$.
However if $\D_{\M} \neq \{ 0\}$ then we can choose in Definition \ref{WF-M Def1} the constant sequence 
$u_N=\varphi u$ for some $\varphi\in\D_{\M}(\Omega)$ with $\varphi(x_0)=1$
and \eqref{WF-M Estimate1} is equivalent to 
\begin{equation*}
\exists C,Q>0\quad\bigl\lvert\widehat{\varphi u}(\xi)\bigr\rvert\leq C \inf_{N}Q^N M_N\lvert\xi\rvert^{-N}\qquad \forall\xi\in\Gamma
\end{equation*}
thus \ref{con} implies
\begin{equation*}
\bigl\lvert\widehat{\varphi u}(\xi)\bigr\rvert\leq C\tilde{h}_\M\biggl(\frac{Q}{\lvert\xi\rvert}\biggr)
\leq C\exp\biggl(-\omega_\M\biggl(\frac{\lvert\xi\rvert}{Q}\biggr)\biggr).
\end{equation*}
We conclude that (c.f. e.g. Rodino \cite{MR1249275} in the case of Gevrey-classes) that for non-quasianalytic weight sequences $\M$ 
\eqref{WF-M Estimate1} is equivalent to 
\begin{equation*}
\exists Q>0 \quad\sup_{\xi\in\Gamma}e^{\omega_\M(Q\lvert\xi\rvert)}\bigl\lvert\widehat{\varphi u}(\xi)\bigr\rvert<\infty.
\end{equation*}
Proposition \ref{MCharFT} is then only a restatement to the well-known fact that for non-quasianalytic 
weight sequences we have that $\varphi\in\D_\M$ if and only if 
$\hat{\varphi}\leq C e^{-\omega_\M(Q\lvert\xi\rvert)}$ for some constants $C, Q$.
Therefore it is possible to define ultradifferentiable classes using appropriately defined weight functions 
instead of weight sequences, see e.g.\ in a somehow generalized setting Bj{\"o}rk \cite{MR0203201}.
However, this approach leads only to non-quasianalytic spaces. 
This restriction was removed by Braun-Meise-Taylor \cite{MR1052587} who reformulated the defining estimates of these classes to 
allow also quasianalytic classes. 
A wavefront set relative to these classes was introduced in Albanese-Jornet-Oliaro \cite{MR2595651}, c.f.\ section \ref{sec:elliptic}.
The complicated connection between the classes defined by weight sequences 
and those given by weight functions was investigated in Bonet-Meise-Melikhov \cite{MR2387040}.
Recently a new approach to define spaces of ultradifferentiable functions was introduced in Rainer-Schindl \cite{MR3285413},
which encompasses the classes given by weight sequences and weight functions, see also Rainer-Schindl \cite{MR3462072}.
\end{Rem}
\section{Invariance of the wavefront set under ultradifferentiable mappings}\label{sec:BV}
Our aim in this section is to develop, using the almost-analytic extension of functions in $\E_\M$ given by Dyn'kin, 
a geometric description of $\WF_\M$ similarly to the one that was presented e.g.\
by Liess \cite[section 4]{MR1806500} for the smooth wavefront set. 

We need to fix some notations: If $\Gamma\subseteq\R^d$ is a cone and $r>0$ then 
\begin{equation*}
\Gamma_r :=\bigl\{y\in \Gamma \mid\, \lvert y\rvert <r\bigr\}.
\end{equation*} 
If $\Gamma^{\prime}\subseteq\Gamma$ is also a cone we write 
$\Gamma^{\prime}\subset\subset\Gamma$ iff 
$(\Gamma^{\prime}\cap S^{d-1})\subset\subset (\Gamma\cap S^{d-1})$.

Analogous to Liess \cite[section 2.1]{MR1806500} 
in the smooth category 
we say that, if $\M$ is a weight sequence, a function 
$F\in\E(\Omega\times U\times\Gamma_r)$, $U\subseteq\R^d$ open, is 
\emph{$\M$-almost analytic} in the variables $(x,y)\in U\times \Gamma_r$ with parameter 
$x^\prime\in\Omega$ iff for all  $K\subset\subset\Omega$, 
$L\subset\subset U$ and cones $\Gamma^{\prime}\subset\subset\Gamma$ there are constants 
$C,Q>0$ such that for some $r^\prime$ we have
\begin{equation}
\biggl\lvert\frac{\partial F}{\partial \bar{z}_j}(x^\prime,x,y)\biggr\rvert \leq Ch_\M (Q\lvert y\rvert) 
\qquad (x^\prime,x,y)\in K\times L\times\Gamma^{\prime}_{r^\prime},\;j=1,\dotsc,d
\end{equation}
where $\tfrac{\partial}{\partial \bar{z}_j}=\tfrac{1}{2}(\partial_{x_j}+i\partial_{y_j})$ 
and $h_\M$ is  the weight associated to the regular weight sequence $\M$ as defined by \eqref{Def:weight}.

We may also say generally that a function $g\in\CC (\Omega\times U\times\Gamma_r)$ 
is of \emph{slow growth} in $ y\in\Gamma_r$ 
if for all $K\subset\subset\Omega$, $L\subset\subset U$ and $\Gamma^\prime\subset\subset\Gamma$
there are constants $c,k>0$ such that
\begin{equation}\label{temperategrowth}
\lvert g(x^\prime,x,y)\rvert\leq c \lvert y\rvert ^{-k} \qquad (x^\prime,x,y)\in K\times L\times
\Gamma^{\prime}_r.
\end{equation}
The next theorem is a generalization of \cite[Theorem 4.4.8]{MR1996773}.
\begin{Thm}\label{Theorem-M-BVWF}
Let $F\in \E(\Omega\times U\times \Gamma_r)$ be $\M$-almost analytic in the variables $(x,y)\in U\times\Gamma_r$
and of slow growth in the variable $y\in\Gamma_r$.
Then the distributional limit $u$ of the sequence 
$u_\eps =F(\,.\,,\,.\,,\eps)\in\E(\Omega\times U)$
exists. We say that $u=b_\Gamma (F)\in\D^\prime(\Omega\times U)$ is the boundary value of $F$.
Furthermore, we have
\begin{equation*}
\WF_\M u\subseteq\,\bigr(\Omega\times U\bigr)\times\bigl(\R^n\times\Gamma^\circ\bigr)
\end{equation*}
where $\Gamma^\circ=\{\eta\in\R^d\mid \langle y,\eta\rangle\geq 0 \;\;\forall y\in\Gamma\}$
is the dual cone of $\Gamma$ in $\R^d$.
\end{Thm}
\begin{proof}
Let $\varphi\in\D(\Omega\times U)$ and $Y_0\in\Gamma_\delta$. 
Then there are $K\subset\subset\Omega$, $L\subset\subset U$ such that 
$\supp \varphi \subseteq K\times L$ and constants $c,k>0$ exists such that \eqref{temperategrowth} holds.
We set 
\begin{equation*}
\Phi_\kappa (x^\prime, x,y)=\sum_{\lvert\alpha\rvert\leq \kappa}\partial^\alpha_x\varphi (x^\prime,x)\frac{(iy)^\alpha}{\alpha !}
\end{equation*}
for $\kappa\geq k$. Obviously $F\cdot\Phi_\kappa$ can be extended to a smooth function on 
$\R^n\times\R^d\times\Gamma_\delta$ that vanishes outside $K\times L\times\Gamma_\delta$. 
We consider the function
\begin{equation*}
u_\eps:\, \R^2\ni (\sigma,\tau)\longmapsto F(x^\prime,\tilde{x}+\sigma Y_0,\eps+\tau Y_0)
\Phi_\kappa(x^\prime,\sigma Y_0,\tau Y_0)
\end{equation*}
where $ x^\prime\in\R^n,\, \tilde{x}\in Y_0^\bot=\{z\in\R^d\mid \langle z,Y_0\rangle=0\}$.
If $a<b$ are chosen such that $\varphi (x^\prime,\tilde{x}+ \sigma Y_0)=0$ 
for all $x^\prime\in\R^n$, $\tilde{x}\in Y_0^\bot$
and $\sigma\leq a$ or $\sigma\geq b$
then $u_\eps (\sigma ,\tau)=0$ for all $\tau\in [0,1]$.
If $R=[a,b]\times [0,1]$ then Stokes' Theorem states that
\begin{equation}\label{Stokes1}
\int_{\partial R} \!u_\eps\; d\zeta=\int_R\! \frac{\partial u_\eps}{\partial \bar{\zeta}} \;d\bar{\zeta}\wedge d\zeta
\end{equation}
where we have set $\zeta =\sigma +i\tau$. 

A simple computation gives
\begin{equation*}
2i\frac{\partial}{\partial \bar{\zeta}}\bigl(\Phi_\kappa(x^\prime,\tilde{x}
+\sigma Y_0,\tau Y_0)\bigl)
=(\kappa +1)\tau^\kappa\sum_{\lvert\alpha\rvert=\kappa +1}\!\!\! \partial_x^\alpha\varphi(x^\prime,\tilde{x}+\sigma Y_0)
\frac{(iY_0)^\alpha}{\alpha !}.
\end{equation*} 
Hence  formula \eqref{Stokes1} means in detail that
\begin{equation*}
\begin{split}
\int_a^b\! F(x^\prime,\sigma Y_0,\eps)\varphi(x^\prime,\sigma Y_0)\,d\sigma 
&=\int_a^b\! F(x^\prime,\sigma Y_0,\eps + Y_0)\Phi_\kappa (x^\prime,\sigma Y_0, Y_0)\,d\sigma\\
&+2i\int_a^b\!\!\int_0^1\!\langle \bar{\partial}F(x^\prime,\sigma Y_0,\eps+ \tau Y_0),Y_0\rangle\Phi_\kappa (x^\prime,\sigma Y_0, \tau Y_0)\,d\tau d\sigma\\
&+(\kappa +1)\int_a^b\!\!\int_0^1\!F(x^\prime,\sigma Y_0,\eps+\tau Y_0)\tau^\kappa\sum_{\lvert\alpha\rvert=\kappa +1}\frac{\partial_x^\alpha \varphi}{\beta !}\,d\tau d\sigma
\end{split}
\end{equation*}
and thus integrating over $\Omega\times Y_0^\bot$ yields
\begin{equation}\label{ConvergenceBVeq}
\begin{split}
\int_{\Omega\times U}\negthickspace\!\! F(x^\prime,x,\eps)\varphi(x^\prime,x)\,d\lambda(x^\prime,x)
&=\int_{\Omega\times U}\!\! F(x^\prime,x,\eps+Y_0)\Phi_\kappa (x^\prime,x,Y_0)\,d\lambda(x^\prime,x)\\
&+2i\! \int_{\Omega\times U}\!\int_0^1 \!\bigl\langle \bar{\partial} F(x^\prime,x,\eps+\tau Y_0),Y_0\bigr
\rangle \Phi_\kappa(x^\prime,x,\tau Y_0)\,d\tau d\lambda(x^\prime,x)\\
&+(\kappa+1)\negmedspace\!\int_{\Omega\times U}\!\int_0^1\! F(x^\prime,x,\eps+\tau Y_0)\tau^\kappa\negthickspace\!\!
\sum_{\lvert\alpha\rvert=\kappa+1}\negthickspace\!\!
\partial_{x}^\alpha \varphi(x^\prime,x)\frac{(iY_0)^\alpha}{\alpha !}d\lambda(x^\prime,x).
\end{split}
\end{equation}
Since by assumption $\lvert\tau^\kappa F(x^\prime,x,\eps+\tau Y_0)\rvert\leq c$ for some constant $c$ and 
$\bar{\partial}_j F$ decreases rapidly for $\Gamma_r\ni y\rightarrow 0$ 
(c.f.\ the remarks after Lemma \ref{lem-weight-con})
the bounded convergence theorem implies that the right-hand side converges for $\eps\rightarrow 0$.
Hence we define
\begin{equation}\label{BVDef}
\begin{split}
\langle u,\varphi\rangle 
&:=\int_{\Omega\times U}\!\! F(x^\prime,x,Y_0)\Phi_\kappa (x^\prime,x,Y_0)\,d\lambda(x^\prime,x)\\
&+2i\! \int_{\Omega\times U}\!\int_0^1 \!\bigl\langle \bar{\partial} F(x^\prime,x,\tau Y_0),Y_0\bigr
\rangle \Phi_\kappa(x^\prime,x,\tau Y_0)\,d\tau d\lambda(x^\prime,x)\\
&+(\kappa+1)\!\int_{\Omega\times U}\!\int_0^1 F(x^\prime,x,\tau Y_0)\tau^\kappa\!\!\!
\sum_{\lvert\alpha\rvert=\kappa+1}\!\!\partial_{x}^\alpha \varphi(x^\prime,x)
\frac{(iY_0)^\alpha}{\alpha !}\,d\tau d\lambda(x^\prime, x).
\end{split}
\end{equation}
Since there is a constant $\tilde{C}$ only depending on $F$ and $K\times L$ such that
\begin{equation*}
\lvert\langle u,\varphi \rangle\rvert\leq \tilde{C}\sup_{(x^\prime ,x)\in K\times L}
\Biggl(\sum_{\lvert\beta\rvert\leq \kappa +1}\bigl\lvert\partial^\beta_x \varphi(x^\prime,x)\bigr\rvert\Biggr)
\end{equation*}
we deduce that the linear form $u$ on $\D(\Omega\times U)$ given by \eqref{BVDef} is a distribution.

Now, let $p_0\in\Omega\times U$ and 
$\omega_2\times V_2\subset\subset\omega_1\times V_1\subset\subset\Omega\times U$ 
two open neighbourhoods of $p_0$. Using \cite[Theorem 1.4.2]{MR1996773} we can choose a sequence
$(\varphi_\kappa)_\kappa\subset\D(\omega_1\times V_1)$ such that 
$\varphi_\kappa\vert_{\omega_2\times V_2}\equiv 1$ and for all $\gamma\in\N_0^{n+d}$ we have that
\begin{equation}\label{Testfunctionest1}
\bigl\lvert D^{\gamma+\beta}\varphi_\kappa\bigr\rvert\leq 
\bigr(C_\gamma(\kappa+1)\bigr)^{\lvert\beta\rvert}\qquad
\lvert\beta\rvert\leq \kappa +1
\end{equation}
for a constant $C_\gamma\geq 1$ independent of $\kappa$.
As before we set for each $\kappa$
\begin{equation*}
\Phi_\kappa (x^\prime,x,y)=\sum_{\lvert\alpha\rvert\leq \kappa}\partial^\alpha_x
\varphi_\kappa (x^\prime,x)\frac{(iy)^\alpha}{\alpha !}.
\end{equation*}
We aim to estimate $\widehat{\varphi_\kappa u}$. 
In order to do so let $(\xi,\eta)\in\R^n\times\R^d$ and notice that \eqref{BVDef} implies 
for $\kappa \geq k$
\begin{equation*}
\begin{split}
\widehat{\varphi_\kappa u}(\xi,\eta)
&=\Bigl\langle u,e^{-i\langle \,.\,,(\xi,\eta)\rangle}\varphi_\kappa\Bigr\rangle\\ 
&=\int_{\Omega\times U}\!\!\! F(x^\prime,x,Y_0)e^{-i(x^\prime\xi +(x+iY_0)\eta)}
\Phi_\kappa (x^\prime,x,Y_0)\,d\lambda(x^\prime,x)\\
&+2i\! \int_{\Omega\times U}\!\int_0^1 \!\bigl\langle \bar{\partial} F(x^\prime,x,\tau Y_0),Y_0\bigr
\rangle e^{-i(x^\prime\xi+(x+i\tau Y_0)\eta)}\Phi_\kappa (x^\prime,x,\tau Y_0)\,d\tau 
d\lambda(x^\prime,x)\\
&+(\kappa+1)\!\int_{\Omega\times U}\!\int_0^1 F(x^\prime,x,\tau Y_0)
e^{-i(x^\prime\xi+(x+i\tau Y_0)\eta )}\tau^\kappa\!\!\!\sum_{\lvert\alpha\rvert=\kappa+1}\!\!
\partial_{x}^\alpha \varphi(x^\prime,x)\frac{(iY_0)^\alpha}{\alpha !}\,d\tau d\lambda (x^\prime,x)
\end{split}
\end{equation*}
for some fixed, but arbitrary $Y_0\in\Gamma_r$ 
(note that $k$ depends on $u$, $\omega_1\times V_1$ and $Y_0$).
Condition \eqref{Testfunctionest1} gives the following estimate for $0\leq\mu\leq\kappa +1$
\begin{equation*}
\Biggl\lvert\sum_{\lvert\alpha\rvert=\mu}\partial^\alpha_{x}\varphi_\kappa(x^\prime,x)
\frac{(iY)^\alpha}{\alpha !}\Biggr\rvert\leq C_0^\mu (\kappa+1)^\mu\sum_{\lvert\alpha\rvert =\mu}
\frac{\lvert Y^\alpha\rvert}{\alpha !}=C_0^\mu (\kappa+1)^\mu\frac{\lvert Y\rvert_1^\mu}{\mu !}
\end{equation*}
where $\lvert Y\rvert_1=\sum_j\lvert Y_j\rvert$ for $Y=(Y_1,\dotsc,Y_d)\in\R^d$. 
Hence we have
\begin{gather*}
\bigl\lvert\Phi_\kappa(x^\prime,x,\tau Y_0)\bigr\rvert \leq C_1^{\kappa +1}\\
\biggl\lvert (\kappa +1)\sum_{\lvert\alpha\rvert =\kappa +1}\partial_x^\alpha\varphi_\kappa 
(x^\prime,x)\frac{(i Y_0)^\alpha}{\alpha !}\biggr\rvert \leq C_1^{\kappa +1}
\end{gather*}
for $C_1=2e^{C_0\lvert Y_0\rvert_1}$ and $\tau\in [0,1]$. We obtain
\begin{equation*}
\begin{split}
\lvert\widehat{\varphi_\kappa u}(\xi,\eta)\rvert &\leq C_1^{\kappa +1}e^{\eta Y_0}
+2C_1^{\kappa +1}C\int_0^1\! h_\M (Q\tau \lvert Y_0\rvert)e^{\tau\eta Y_0}\,d\tau
+C_1^{\kappa +1}\int_0^1\!\tau^{\kappa -k} e^{\tau\eta Y_0}\,d\tau\\
&\leq C_2 Q_1^\kappa\biggl( e^{\eta Y_0}+m_{\kappa -k}\int_0^1\tau^{\kappa -k} e^{\eta Y_0}\biggr)
=C_2Q_1^\kappa\Bigl( e^{\eta Y_0}+m_\kappa(\kappa -k)!(-Y_0\eta)^{k-\kappa -1}\Bigr)
\end{split}
\end{equation*}
for some constants $C_2,Q_1$ and $Y_0\eta<0$. 
 If we set $\tilde{Y}_0=(0,Y_0)\in \R^n\times\R^d$ then obviously
 \begin{equation*}
 \bigl\langle \tilde{Y}_0,(\xi,\eta)\bigr\rangle=\langle Y_0, \eta\rangle.
 \end{equation*} 
 Therefore we have for $\kappa\geq k$ and $\zeta =(\xi,\eta)$ that
 \begin{equation*}
 \bigl\lvert\widehat{\varphi_\kappa u}(\zeta)\bigr\rvert=C_3 Q_1^\kappa\biggl(e^{\tilde{Y}_0\zeta}
 +m_{\kappa -k}(\kappa -k)! \bigl(-\tilde{Y}_0\zeta\bigr)^{k-\kappa -1}\biggr)
 \end{equation*}
and $\tilde{Y}_0\zeta<0$.
 
 Now for any $\zeta_0\in\R^{n+d}$ with $\langle \tilde{Y}_0,\zeta_0\rangle<0$
  we can choose an open cone $V\subseteq\R^{n+d}$ such that $\zeta_0\in V$ and 
  for some constant $c >0$
  we have $\langle \tilde{Y}_0,\zeta\rangle <-c\lvert\zeta\rvert$ if $\zeta\in V$.
  Furthermore we set $u_\kappa =\varphi_{k+\kappa -1}u$.
 Clearly the sequence $(u_\kappa)_\kappa$ is bounded in $\E^{\prime}(\Omega\times U)$ and
 $u_\kappa\vert_{\omega_2\times V_2}\equiv u\vert_{\omega_2\times V_2}$.
Also using the inequality $e^{-c\lvert\zeta\rvert}\leq \kappa !(c\lvert\zeta\rvert)^{-\kappa}$ we conclude 
\begin{equation*}
 \bigl\lvert\hat{u}_\kappa(\zeta)\bigr\rvert=C_3 Q_1^\kappa\biggl(\kappa ! 
 (c\lvert\zeta\rvert)^{-\kappa}
 +m_{\kappa -1}(\kappa -1)! \bigl(c\lvert\zeta\rvert\bigr)^{-\kappa}\biggr)
 \leq C_3 Q_2^\kappa m_\kappa \kappa !\lvert\zeta\rvert^{-\kappa} \qquad \zeta \in V.
 \end{equation*}
 Hence $(p_0,\zeta_0)\notin \WF_\M u$ and therefore
 \begin{equation*}
 \WF_\M u\subseteq \bigl(\Omega\times U\bigr)\times \bigl(\R^n\times\Gamma^\circ\bigr)\!
 \setminus\!\{(0,0)\}.
 \end{equation*}
 \end{proof}
 It is clear that the proof would only require $F\in\CC^1$. From now the constants used in the proofs  will be generic, i.e.\ they may change from line to line.
 \begin{Rem}
If $F\in\E (\Omega\times U\times V)$ is $\M$-almost analytic with respect to the variables $(x,y)\in U\times V$ 
we will often write $F(x^\prime, x+iy)$ or $F(x^\prime, z,\bar{z})$ and consider $F$ as a smooth function
on $\Omega\times(U+iV)$.
If $\Omega=\emptyset$ then we just say that $F$ is $\M$-almost analytic.

Even though in the remainder of this paper we shall 
 only use the assertion of Theorem \ref{WF-M-Micro-decomp} 
 in the special case $\Omega=\emptyset$ (i.e. without 
 parameters), we have decided to include the general 
 statement because we think it is of independent interest. We 
 also have an application for the parameter version 
 of the theorem in our upcoming paper \cite{Furdos2}. 

 \end{Rem}
 \begin{Ex}\label{ExJump}
Consider the holomorphic function $F(z)=\tfrac{1}{z}$ on $\C\!\setminus\!\{0\}$. 
It is well known that the boundary values of $F$ onto the real line from above and beneath, commonly denoted by
\begin{align*}
\frac{1}{x+i0}&=b_+F=\lim_{y\rightarrow 0+}\frac{1}{x+iy}\\
\frac{1}{x-i0}&=b_+F=\lim_{y\rightarrow 0+}\frac{1}{x-iy}\\
\intertext{satisfy the jump relations (c.f.\ e.g.\ Duistermaat-Kolk \cite{MR2680692}), in particular}
2i\delta &=\frac{1}{x-i0}-\frac{1}{x+i0}.
\end{align*}
We have that both $\tfrac{1}{x+i0}$ and $\tfrac{1}{x-i0}$ are real-analytic outside the origin.
Hence the application of Theorem \ref{Theorem-M-BVWF} together with the jump relations imply that
\begin{equation*}
\WF_\M \biggl(\frac{1}{x\pm i0}\biggr)=\{0\}\times\R_\pm.
\end{equation*}
 \end{Ex}

 There is a partial converse to the last theorem.
\begin{Thm}\label{BV-M-WF}
Let $\Gamma\subseteq \R^n$ be an open convex cone and $u\in\D^\prime(\Omega)$
with $\WF_\M u\subseteq\Omega\times\Gamma^\circ$. 
If $V\subset\subset\Omega$ and $\Gamma^\prime$ is an open convex cone 
with $\overline{\Gamma}^\prime\subseteq\Gamma\cup\{0\}$ then there is an $\M$-almost analytic function $F$ on 
$V+i\Gamma^\prime_r$ of slow growth for some $r>0$ such that $u\vert_V=b_{\Gamma^\prime}(F)$
\end{Thm}
\begin{proof}
By \cite[Theorem 8.4.15]{MR1996773} we have that $u$ can be written on a bounded neighbourhood $U$ of 
$V$ as a sum of a function $f\in\E_\M(U)$ and the boundary value of a holomorphic function of slow growth 
on $U+i\Gamma^\prime_r$ for some $r$.
 To obtain the assertion use Corollary \ref{CharMalmostanalytic} to extend $f$ almost-analytically on $V$. 
\end{proof}
In order to proceed we need a further refinement of a result of H{\"o}rmander.
\begin{Lem}\label{WF-M-Micro-decomp}
Let $\Gamma_j\subseteq\R^n\!\setminus\!\{0\}$, $j=1,\dots, N$, be closed cones such that
$\bigcup_j\Gamma_j=\R^n\!\setminus\!\{0\}$ and $V\subset\subset\Omega$ convex. 
Any $u\in\D^\prime (\Omega)$ can be written on $V$ as a linear combination $u\vert_V=\sum_j u_j$
of distributions $u_j\in\D^{\prime}(V)$ that satisfy
\begin{equation*}
\WF_\M u_j\subseteq \WF_\M u\cap \bigl(V\times\Gamma_j\bigr)
\end{equation*}
\end{Lem}
\begin{proof}
Set $v=\varphi u$ where $\varphi\in\D(\Omega)$ such that $\varphi\equiv 1$ on $V$. 
\cite[Corollary 8.4.13]{MR1996773} gives the existence of $v_j\in\mathcal{S}^\prime(\R^n)$ such that
\begin{equation*}
\WF_\M v_j\subseteq \WF_\M v\cap \bigl(\R^n\times\Gamma_j\bigr).
\end{equation*}
Set $u_j=(v_j)\vert_U$.
\end{proof}
Combining Theorem \ref{BV-M-WF} with Lemma \ref{WF-M-Micro-decomp} we obtain
\begin{Cor}\label{WF-M-localdescr1}
Let $u\in\D^{\prime}(\Omega)$ and $(x_0,\xi_0)\in\Omega\times\R^n\!\setminus\!\{0\}$. 
Then $(x_0,\xi_0)\notin\WF_\M u$ if and only if there are a neighbourhood $U$ of $x_0$,
open convex cones $\Gamma_1,\dots,\Gamma_N$ with the properties 
$\xi_0\Gamma_j<0$, $j=1,\dots N$ and $\Gamma_j\cap\Gamma_k=\emptyset\text{ for}\;j\neq k$,
and $\M$-almost analytic functions $h_j$ on $U+i\Gamma_{r_j}$, $r_j>0$, of slow growth
such that
\begin{equation*}
u\vert_U=\sum_{j=1}^N b_{\Gamma_j}(h_j)
\end{equation*}
\end{Cor}
\begin{proof}
W.l.o.g.\ assume that $\WF_{\M} u\neq\emptyset$. If $(x_0,\xi_0)\notin\WF_\M u$ one can find 
closed cones $V_1,\dots,V_N$ with nonempty interior and $V_j\cap V_k$ has measure zero for $j\neq k$
 such that $\xi_0$ is contained in the interior of 
$V_1$ and $V_1\cap\WF_{\M}u=\emptyset$ whereas $\xi_0\notin V_j$ are acute cones  and 
 $\WF_\M u\cap V_j\neq\emptyset$ for $j=2,\dotsc, N$.
 By Lemma \ref{WF-M-Micro-decomp} we can write $u$ on an open neighbourhood $U$ of $x_0$ 
 as a sum $u=u_1+\sum_{j=2}^N u_j$ with $u_1$ being an ultradifferentiable function defined on $U$
 and $u_j\in\D^{\prime}(U)$ such that $\WF_\M u_j\subseteq \WF_\M u\cap V_j$, $j=2,\dotsc ,N$.
The cones $V_2,\dotsc,V_N$ are the dual cones of open convex cones $\Gamma_2,\dotsc,\Gamma_N$,
 i.e.\ $\Gamma_j^\circ=V_j$. We can choose cones $\Gamma_j^\prime\subset\subset\Gamma_j$ 
 and using Theorem \ref{BV-M-WF} we find $\M$-almost analytic functions $h_j$ on 
 $U +i\Gamma^\prime_{j,r}$ of slow growth such that $u_j=b_{\Gamma^\prime_j}(h_j)$. 
It remains to note that $\xi_0y<0$ for all $y\in\Gamma^{\prime}_j$, $j=2,\dotsc, N$.
\end{proof}

Let $\Omega_1\subseteq\R^m$ and $\Omega_2\subseteq\R^n$ be open. 
If $F:\Omega_1\rightarrow\Omega_2$ is a $\E_\M$-mapping then we denote as in \cite[page 263]{MR1996773}
the set of normals by 
\begin{equation*}
N_F=\bigl\{(F(x),\eta)\in \Omega_2\times\R^n:\; DF(x)\eta=0\bigr\}.
\end{equation*}
where $DF$ denotes the transpose of the Jacobian of $F$.
The following is a generalization of \cite[Theorem 8.5.1]{MR1996773}
\begin{Thm}\label{WF-M-composition}
For any $u\in\D^{\prime}(\Omega_2)$ with $N_F\cap\WF_\M u=\emptyset$ we obtain that the pull-back 
$F^*u\in\D^{\prime}(\Omega_1)$ is well defined and
\begin{equation}
\WF_\M\bigl(F^*u\bigr)\subseteq F^*\bigl(\WF_\M u\bigr).
\end{equation}
\end{Thm}
\begin{proof}
The first part of the statement is \cite[Theorem 8.2.4]{MR1996773}.
For the proof of the second part of the theorem assume first that there is an open convex cone $\Gamma$ 
such that $u$ is the boundary value of an $\M$-almost analytic function $\Phi$ on $
\Omega_2+i\Gamma_r$ of slow growth. Hence $\WF_\M u\subseteq\Omega_2\times\Gamma^\circ$.
 If $x_0\in\Omega_1$ and $DF(x_0)\eta\neq 0$ for $\eta\in\Gamma^\circ\!\setminus\!\{0\}$ then
 $DF(x_0)\Gamma^\circ$ is a closed convex cone. We claim that 
\begin{equation*}
\WF_\M (F^{\ast}u)\vert_{x_0}\subseteq \bigl\{(x_0,DF(x_0)\eta):\; \eta\in\Gamma^\circ\!\setminus\{0\}\!\bigr\}.
\end{equation*}
We adapt as usual the argument of H{\"o}rmander \cite{MR1996773}. We can write (see \cite[page 296]{MR1996773})
\begin{equation*}
DF(x_0)\Gamma^\circ =\bigl\{\xi\in\R^n\,\vert\; \langle h,\xi\rangle\geq 0,\; F^\prime(x_0)h\in\Gamma\bigr\}.
\end{equation*}
If $\tilde{F}$ denotes an $\M$-almost analytic extension of $F$ onto $X_0+i\R^n$, $X_0\in\U(x_0)$ convex and  
relatively compact in $\Omega_1$, which exists due to Theorem \ref{Dynkin1}, 
then Taylor's formula implies that
\begin{equation*}
\imag \tilde{F}(x+i\eps h)\in \Gamma\qquad x\in X_0
\end{equation*}
for $F^{\prime}(x_0)h\in\Gamma$ if $X_0$ and $\eps>0$ are small.

Recalling \eqref{ConvergenceBVeq} we see that the map
\begin{equation*}
\R_{\geq 0}\times\bigl(\Gamma\cup\{0\}\bigr)\ni(\eps, y)\longmapsto \tilde{\Phi}(\eps,y):=\Phi\bigl(\tilde{F}(\,.\,+i\eps h)+iy\bigr)\in\D^{\prime}(X_0)
\end{equation*}
is continuous.
 If $\eps \rightarrow 0$ then $\tilde{\Phi}(\eps,y)\rightarrow \tilde{\Phi}(0,y)=\Phi(\tilde{F}(\,.\,+0i)+iy)$
 in $\D^\prime$ and if now $y\rightarrow 0$ we have by definition $\tilde{\Phi}(0,y)\rightarrow F^*u$.
 On the other hand if first $y\rightarrow 0$ then 
 $\tilde{\Phi}(\eps,y)\rightarrow \tilde{\Phi}(\eps,0)=\Phi(\tilde{F}(\,.\,+i\eps h))$. 
 Hence by continuity
 \begin{equation*}
 F^\ast u=\lim_{\eps\rightarrow 0} \Phi\bigl(\tilde{F}(\,.\,+i\eps h)\bigr)
 \end{equation*}
 in $\D^{\prime}(X_0)$ and by the proof of Theorem \ref{Theorem-M-BVWF}
 \begin{equation*}
 \WF_\M F^\ast u\vert_{x_0}\subseteq \{(x_0,\xi)\,\vert\; \langle h,\xi\rangle\geq 0\}.
 \end{equation*}
 The claim follows.
 
Now suppose that $(F(x_0),\eta_0)\notin\WF_\M u$.
By Corollary \ref{WF-M-localdescr1}
 we can write a general distribution 
 $u$ on some neighbourhood $U_0$ of $F(x_0)$ as  $\sum_{j=1}^N u_j$ 
 where the distributions $u_j$, $j=1,\dots,N$, are the boundary values of some $\M$-almost analytic functions
 $\Phi_j$ on $U_0+i\Gamma_j$, where the $\Gamma_j$ are some open convex cones such that
  $\eta_0\Gamma_j<0$ for all $j=1,\dotsc,N$.
By  assumption $DF(x)\eta\neq 0$ when $(F(x),\eta)\in\WF_\M u$ for $x\in F^{-1}(U_0)$.
Hence we can assume that $DF(x)\eta\neq 0$ for $\eta\in\Gamma_j^\circ$ for all $j=1,\dotsc, N$ 
and $x\in F^{-1}(U_0)$ since in the proof of Corollary \ref{WF-M-localdescr1} the cones 
$\Gamma_j$, $j=1,\dotsc,N$, can be chosen such that $\Gamma^\circ\cap S^{n-1}$
and $\Gamma^\circ_j\cap \WF_\M u\vert_x\neq\emptyset$ for $x\in U_0$. 
By the arguments above we have for a small neighbourhood $V$ of $x_0$ that
\begin{equation*}
F^{\ast}u\vert_V=\sum_{j=1}^N{F^\ast u_j}\vert_V
\end{equation*}
and 
$\WF_\M (F^\ast u_j)\vert_{x_0}\subseteq \{(x_0,DF(x_0)\eta)\mid \eta\in\Gamma_j^\circ\!\setminus\!\{0\}\}$
for all $j$.
However, since $\eta_0\Gamma_j<0$ it follows that $(x_0,DF(x_0)\eta_0)\!\notin\!\WF_\M (F^\ast u_j)$ 
and therefore $(x_0,DF(x_0)\eta_0)\!\notin\!\WF_\M(F^\ast u)$.
\end{proof}
\begin{Rem}
If $F$ is an $\E_\M$-diffeomorphism we obtain from Theorem \ref{WF-M-composition} that
\begin{equation*}
\WF_\M \bigl(F^*u\bigr) =F^*\bigl(\WF_\M u\bigr).
\end{equation*}

Hence if $M$ is an $\E_\M$-manifold and $u\in\D^{\prime}(M)$ we can define $\WF_{\M} u$ invariantly
as a subset of $T^{\ast}M\!\setminus\!\{0\}$.
More precisely,  there is a subset $K_u$ of $T^\ast M$ such that the diagram
\begin{equation*}
\begin{tikzcd}
& K_u\arrow[dl]\arrow[dr]\\
T^\ast \varphi (U\cap V)\supseteq\WF_\M v_1\arrow[rr,"\rho^\ast"]&&\WF_\M v_2\subseteq T^\ast \psi(U\cap V)
\end{tikzcd}
\end{equation*}
commutes for any two charts $\varphi$ and $\psi$ of $M$ on $U\subseteq M$ and $V\subseteq M$, respectively. 
We have set $\rho=\psi\circ\varphi^{-1}$, $v_1=\varphi^\ast u\in\D^{\prime}(\varphi(U\cap V))$ and 
$v_2=\psi^\ast u\in\D^{\prime}(\psi(U\cap V))$.
It follows that $K_u\subseteq T^\ast M\!\setminus\!\{0\}$ has to be closed and fiberwise conic.
We set $\WF_\M u:=K_u$.

Analogously we define the wavefront set of a distribution $u\in\D^{\prime}(M,E)$ with values in an ultradifferentiable vector bundle locally by setting
\begin{equation*}
\WF_\M u\vert_V=\bigcup_{j=1}^\nu u_j
\end{equation*}
where $V\subseteq M$ is an open subset such that there is a local basis $\omega^1,\dotsc,\omega^\nu$ of $\E_\M(V,E)$ and $u_j\in\D^\prime(V)$ are distributions on $V$ such that
\begin{equation*}
u\vert_V=\sum_{j=1}^\nu u_j\omega^j.
\end{equation*}
\end{Rem}

We close this section by observing that Theorem \ref{WF-M-composition} allows us to strengthen a uniqueness result of Boman \cite{MR1382568}:
\begin{Thm}
Let $\M$ be a quasianalytic weight sequence and $S\subseteq\R^n$ an $\E_\M$-submanifold. 
If $u$ is a distribution defined on a neighbourhood of $S$ such that
\begin{gather*}
\WF_\M u\cap N^{\ast}S=\emptyset\\
\intertext{and}
\partial^\alpha u\vert_S=0\qquad \forall \alpha\in\N_0^n,
\end{gather*}
then $u$ vanishes on some neighbourhood of $S$.
\end{Thm}
Indeed, locally $S$ is diffeomorphic to
\begin{equation*}
S^\prime=\bigl\{(x^\prime,x^{\prime\prime})\in\R^{m+d}\mid x^{\prime\prime}=0\bigr\}\subseteq\R^n
\end{equation*}
and the assumptions of the theorem translate to the corresponding conditions for the pullback $w=F^\ast u$ 
where $F:\R^n\rightarrow \R^n$ is the local $\E_\M$-diffeomorphism that maps $S^\prime$ to $S$.
Then the proof of Theorem 1 in \cite{MR1382568} gives $w=0$ in a neighbourhood of $S^\prime$.
\section{A generalized version of Bony's Theorem}\label{sec:Bony}
We have seen that for a distribution $u$ the wavefront set $\WF_\M u$ can be described either using the 
Fourier transform or by its $\M$-almost analytic extensions. 
The similar fact is true for the analytic wavefront set using holomorphic extensions. 
The latter was the original approach of Sato \cite{zbMATH03331746}.
However, Bros-Iagolnitzer \cite{MR0399494} used the classical
 FBI-Transform to describe the set of microlocal analytic singularities. 
 It was Bony \cite{MR0650834} who proved
 that all three methods describe actually the same set.
 In the ultradifferentiable setting Chung-Kim \cite{MR1492944}, see also Kim-Chung-Kim \cite{MR1826922}, used the FBI transform to define
 an ultradifferentiable singular spectrum for Fourier hyperfunctions. 
 However, they did not mention how this singular spectrum in the case of
 distributions may be related to $\WF_\M$ as defined by H{\"o}rmander.
 Our next aim is to show an ultradifferentiable version of Bony's theorem.
 We will work in the generalized setting of Berhanu and Hounie \cite{MR2864805}.
We shall note that recently Berhanu and Hailu \cite{MR3701268} showed that the Gevrey classes can be characterized by this generalized FBI transform and
Hoepfner and Medrado \cite{Hoepfner-MedradoPreprint} also proved a characterization of the ultradifferentiable wavefront set for a certain class of 
non-quasianalytic weight sequences.
 
Let $p$ be a real, homogeneous, positive, elliptic polynomial of degree $2k$, $k\in\N$, on $\R^n$, i.e. 
\begin{equation*}
p(x)=\sum_{\alpha =2k}a_{\alpha}x^\alpha\qquad a_\alpha\in\R,
\end{equation*}
and there are constants $c,C>0$ such that 
\begin{equation*}
c\lvert x\rvert^{2k}\leq p(x)\leq C\lvert x\rvert^{2k}\qquad x\in\R^n.
\end{equation*}
Let $c_p^{-1}=\int e^{-p(x)}dx$. As in \cite[section 4]{MR2864805} we consider the generalized FBI transform
with generating function $e^{-p}$ of a distribution of compact support $u\in\E^\prime(\R^n)$, i.e.\
\begin{equation*}
\mathfrak{F}u(t,\xi)=c_p\Bigl\langle u(x),e^{i\xi(t-x)}e^{-\lvert\xi\rvert p(t-x)}\Bigr\rangle.
\end{equation*}
The inversion formula is
\begin{equation}\label{FBI-inverse}
u =\lim_{\eps\rightarrow\infty}\int_{\R^n\times\R^n}\!\!e^{i\xi(x-t)}e^{-\eps\lvert\xi\rvert^2}
\mathfrak{F}u(t,\xi)\lvert\xi\rvert^{\tfrac{n}{2k}}\,dtd\xi
\end{equation}
where of course the distributional limit is meant.
\begin{Thm}\label{M-FBIThm}
Let $u\in\D^\prime(\Omega)$ and $(x_0,\xi_0)\in T^\ast\Omega\!\setminus\!\{0\}$. 
Then $(x_0,\xi_0)\notin\WF_\M u$ if and only if there is a test function $\psi\in\D(\Omega)$ with 
$\psi\vert_U\equiv 1$ for some neighbourhood $U$ of $x_0$ such that 
\begin{equation}\label{M-FBIestimate}
\sup_{(t,\xi)\in V\times\Gamma}e^{\omega_\M(\gamma\lvert\xi\rvert)}
\bigl\lvert\mathfrak{F}(\psi u)(t,\xi) \bigr\rvert<\infty
\end{equation}
for some conic neighbourhood $V\times\Gamma$ of $(x_0,\xi_0)$ and some constant $\gamma>0$.
\end{Thm}
\begin{proof}
First, assume that $(x_0,\xi_0)\notin \WF_\M u$. By Corollary \ref{WF-M-localdescr1} we know that for
some neighbourhood $U$ of $x_0$
\begin{equation*}
u\vert_U=\sum_{j=1}^Nb_{\Gamma^j}(F_j)
\end{equation*}
where $F_j$ are $\M$-almost analytic on $U\times\Gamma^j_{r_j}$ for cones $\Gamma^j$ with
$\xi_0\Gamma^j<0$. Hence it suffices to prove the necessity of \eqref{M-FBIestimate} for $u=b_\Gamma (F)$
being the boundary value of an $\M$-almost analytic function on $U\times\Gamma_d$ 
where $\Gamma$ is a cone with the property that $\xi_0\Gamma <0$.
 W.l.o.g.\ $x_0=0$ and let $r>0$ such that $B_{2r}(0)\subset\subset U$ and
$\psi\in\D(B_{2r}(0))$ such that $\psi\vert_{B_r(0)}\equiv 1$. 
Furthermore we choose $v\in\Gamma_d$ and set
\begin{equation*}
Q(t,\xi,x)=i\xi(t-x)-\lvert\xi\rvert p(t-x).
\end{equation*}
Then 
\begin{equation*}
\mathfrak{F}(\psi u)(t,\xi)=\lim_{\tau\rightarrow 0+}\int_{B_{2r}(0)}\!e^{Q(t,\xi,x+i\tau v)}\psi(x)F(x+i\tau v)\,dx.
\end{equation*}
As in the proof of Theorem 4.2 in \cite{MR2864805} we  put $z=x+iy$, $\psi(z)=\psi(x)$ and
\begin{equation*}
D_\tau :=\bigl\{ x+i\sigma v\in\C^n\mid x\in B_{2r}=B_{2r}(0),\; \tau\leq\sigma\leq\lambda\bigr\}
\end{equation*}
for some $\lambda>0$ to be determined later
 and consider the $n$-form
\begin{equation*}
e^{Q(t,\xi,z)}\psi(z)F(z)\,dz_1\wedge \dots\wedge dz_n.
\end{equation*}
Since $\psi\in\D(B_{2r}(0))$ Stokes' theorem implies
\begin{equation}\label{FBI-neccesity1}
\begin{split}
\int_{B_{2r}}\!e^{Q(t,\xi,x+i\tau v)}\psi(x)F(x+i\tau v)\,dx 
&=\int_{B_{2r}}\!e^{Q(t,\xi,x+i\lambda v)}\psi(x)F(x+i\lambda v)\,dx\\
&+\sum_{j=1}^n\int_{D_\tau}\! e^{Q(t,\xi,z)}\frac{\partial}{\partial \bar{z}_j}\bigl(\psi(z)F(z)\bigr)
\,d\bar{z}_j\wedge dz_1\wedge\dots\wedge dz_n\\
&=\int_{B_{2r}}\!e^{Q(t,\xi,x+i\lambda v)}\psi(x)F(x+i\lambda v)\,dx\\
&+\sum_{j=1}^n\int_{B_{2r}}\!\int_\tau^\lambda \!e^{Q(t,\xi,x+i\sigma v)}
\frac{\partial\psi}{\partial \bar{z}_j}(x+i\sigma v)F(x+i\sigma v)\,d\sigma dx\\
&+\sum_{j=1}^n\int_{B_{2r}}\!\int_\tau^\lambda \!e^{Q(t,\xi,x+i\sigma v)}\psi(x+i\sigma v)
\frac{\partial F}{\partial \bar{z}_j}(x+i\sigma v)\,d\sigma dx.
\end{split}
\end{equation}

We need to estimate the integrals on the right-hand side of \eqref{FBI-neccesity1}.
Since $\xi_0\cdot v<0$ there is an open cone $\Gamma_1$ containing $\xi_0$ such that 
$\xi\cdot v\leq -c_0\lvert\xi\rvert\lvert v\rvert$ for all $\xi\in \Gamma_1$ and some constant $c_0>0$.
We note that for $\xi\in \Gamma_1$ and $t$ in some bounded neighbourhood of the origin we have
\begin{equation*}
\begin{split}
\real Q(t,\xi,x+i\lambda v)&=\lambda (\xi v)-\lvert\xi\rvert\real p(t-x-i\lambda v)\\
&=\lambda (\xi v)-\lvert\xi\rvert\bigl(\real p(t-x) +O(\lambda^2)\lvert v\rvert^2\bigr)\\
&\leq \lambda(\xi v)-c\lvert\xi\rvert\bigl(\lvert t-x\rvert^{2k}+O(\lambda^2)\lvert v\rvert^2\bigr)\\
&\leq -c_0\lambda\lvert v\rvert\lvert\xi\rvert +O\bigl(\lambda^2\bigr)\lvert\xi\rvert.
\end{split}
\end{equation*}
Hence for $\lambda$ small enough
\begin{equation}\label{exponentialdecay1}
\real Q(t,\xi,x+i\lambda v)\leq -\frac{c_0}{2}\lambda \lvert v\rvert\lvert\xi\rvert
\end{equation}
where $\xi\in \Gamma_1$, $x\in B_{2r}$ and $t$ is in a bounded neighbourhood $V$ of $0$.
 We conclude that
\begin{equation*}
\left\lvert \;\int_{B_{2r}}\!e^{Q(t,\xi,x+i\lambda v)}\psi(x)F(x+i\lambda v)\,dx\;\right\rvert\leq 
C_1e^{-\gamma_1\lvert\xi\rvert}
\end{equation*}
for some constants $\gamma_1,C_1>0$ and $(t,\xi)\in V\times\Gamma_1$.
We note that \eqref{analyticincl} implies that $\omega_\M(t)=O(t)$ for $t\rightarrow \infty$, c.f.\ e.g.\ Komatsu \cite{MR0320743} or Bonet-Meise-Melikhov \cite{MR2387040}, thence 
\begin{equation*}
\left\lvert \;\int_{B_{2r}}\!e^{Q(t,\xi,x+i\lambda v)}\psi(x)F(x+i\lambda v)\,dx\;\right\rvert\leq C_1
e^{-\omega_\M(\gamma_1\lvert\xi\rvert)}
\end{equation*}
for $(t,\xi)\in V\times\Gamma_1$.

On the other hand we can also estimate
\begin{equation*}
\begin{split}
\real Q(t,\xi,x+i\sigma v)&\leq \sigma (\xi v)- c\lvert t-x\rvert^{2k}\lvert\xi\rvert
+ O\bigl(\lambda^2\bigr)\lvert\xi\rvert\\
&\leq -c\lvert t -x\rvert^{2k}\lvert\xi\rvert+ O\bigl(\lambda^2\bigr)\lvert\xi\rvert
\end{split}
\end{equation*}
since $\xi v<0$ for all $\xi\in\Gamma_1$. 
If $x\in\supp (\partial \psi /\partial\bar{z}_j)$ then $\lvert x\rvert\geq r$.
Therefore if $\lvert t\rvert\leq r/2$ and $\lambda$ small enough we obtain that there is a 
constant $\gamma_2>0$ such that
\begin{equation*}
\real Q(t,\xi,x+i\sigma v)\leq -\gamma_2\lvert\xi\rvert
\end{equation*}
for all $\xi\in\Gamma_1$. Hence
\begin{equation*}
\left\lvert\; \sum_{j=1}^n\int_{B_{2r}}\!\int_\tau^\lambda \!e^{Q(t,\xi,x+i\sigma v)}
\frac{\partial\psi}{\partial \bar{z}_j}(x+i\sigma v)F(x+i\sigma v)\,d\sigma dx\; \right\rvert
\leq C_2e^{-\gamma_2\lvert\xi\rvert}\leq C_2e^{-\omega_\M(\gamma_2\lvert \xi\rvert)}
\end{equation*}
for $\xi\in\Gamma_1$, 
$\lvert t\rvert\leq r/2$ and all $0<\tau<\lambda$.

In order to estimate the third integral in \eqref{FBI-neccesity1} we remark that by 
\eqref{exponentialdecay1} we have for a generic constant $C_3>0$ and all $k\in\N_0$ that
\begin{align*}
\left\lvert\; \sum_{j=1}^n\int_{B_{2r}}\!\int_\tau^\lambda \!e^{Q(t,\xi,x+i\sigma v)}\psi(x)
\frac{\partial F}{\partial \bar{z}_j}(x+i\sigma v)\,d\sigma dx\;\right\rvert
&\leq C_3\int_0^\infty\! e^{-c_0\sigma \lvert v\rvert\lvert\xi\rvert}h_\M(\rho\sigma \lvert v\rvert)\;
d\sigma \\
&\leq C_3\int_0^\infty\!e^{-c_0\sigma \lvert v\rvert\lvert\xi\rvert} \rho^k\sigma^k\lvert v\rvert^k m_k\, 
d\sigma\\
&=C_3\rho^km_kc_0^{-k}\lvert\xi\rvert^{-k}k!\\
&=C_3\rho_1^kM_k\lvert\xi\rvert^{-k}.\\
\intertext{Hence by Lemma \ref{lem-weight-con}}
\left\lvert\; \sum_{j=1}^n\int_{B_{2r}}\!\int_\tau^\lambda \!e^{Q(t,\xi,x+i\sigma v)}\psi(x)
\frac{\partial F}{\partial \bar{z}_j}(x+i\sigma v)\,d\sigma dx\;\right\rvert
&\leq C_3\tilde{h}_\M \bigl(\rho_1\lvert\xi\rvert^{-1}\bigr)\\
&\leq C_3e^{-\omega_\M(\rho_1\lvert\xi\rvert)}.
\end{align*}

In view of \eqref{FBI-neccesity1} we have shown that for $\xi\in\Gamma_1$
and $t$ in a small enough neighbourhood of $0$ there are constants $C,\gamma>0$ such that 
\begin{equation*}
\left\lvert\;\int_{B_{2r}}\!e^{Q(t,\xi,x+i\tau v)}\psi(x)F(x+i\tau v)\,dx \right\rvert\leq Ce^{-\omega_\M(\gamma\lvert\xi\rvert)}.
\end{equation*}
Note that in the estimate the constants $C$ and $\gamma$ depend on $\lambda$ but not on $\tau<\lambda$.
Thus \eqref{M-FBIestimate} is proven.

On the other hand, assume that \eqref{M-FBIestimate} holds for a point $(x_0,\xi_0)$, i.e.\
that there is a neighbourhood $V$ of $x_0$, an open cone $\Gamma\subseteq\R^n$ 
containing $\xi_0$ and constants $C,\gamma >0$ such that 
\begin{equation}\label{M-FBIestimate2}
\bigl\lvert\mathfrak{F}(\psi u)(x,\xi)\bigr\rvert\leq Ce^{-\omega_\M(\gamma\lvert\xi\rvert)}
\qquad x\in V,\; \xi\in \Gamma
\end{equation}
for some test function $\psi\in\D(\Omega)$ that is $1$ near $x_0$.
We may assume that $x_0=0$. We have to prove that $(0,\xi_0)\notin \WF_\M u$ or, equivalently, 
$(0,\xi_0)\notin \WF_\M v$ where $v=\psi u$. We invoke the inversion formula \eqref{FBI-inverse} for
the FBI transform
\begin{equation*}
v =\lim_{\eps\rightarrow\infty}\int_{\R^n\times\R^n}\!\!e^{i\xi(x-t)}e^{-\eps\lvert\xi\rvert^2}
\mathfrak{F}v(t,\xi)\lvert\xi\rvert^{\tfrac{n}{2k}}\,dtd\xi
\end{equation*}
and split the occuring integral into 4 parts
\begin{equation}
\int_{\R^n\times\R^n}\!\!e^{i\xi(x-t)}e^{-\eps\lvert\xi\rvert^2}
\mathfrak{F}v(t,\xi)\lvert\xi\rvert^{\tfrac{n}{2k}}\,dtd\xi =I_1^\eps(x) +I_2^\eps(x) +I_3^\eps(x) +I_4^\eps(x)
\end{equation}
where
\begin{align*}
I_1^\eps(x)&=\int_{\R^n}\int_{\lvert t\rvert\leq a}\!\!e^{i\xi(x-t)}e^{-\eps\lvert\xi\rvert^2}
\mathfrak{F}v(t,\xi)\lvert\xi\rvert^{\tfrac{n}{2k}}\,dtd\xi \\
I_2^\eps(x)&=\int_{\lvert\xi\rvert\leq B}\int_{a\leq\lvert t\rvert\leq A}\!\!e^{i\xi(x-t)}
e^{-\eps\lvert\xi\rvert^2}
\mathfrak{F}v(t,\xi)\lvert\xi\rvert^{\tfrac{n}{2k}}\,dtd\xi \\
I_3^\eps(x)&=\int_{\R^n}\int_{\lvert t\rvert\geq A}\!\!e^{i\xi(x-t)}e^{-\eps\lvert\xi\rvert^2}
\mathfrak{F}v(t,\xi)\lvert\xi\rvert^{\tfrac{n}{2k}}\,dtd\xi \\
I_4^\eps(x)&=\int_{\lvert\xi\rvert\geq B}\int_{a\leq \lvert t\rvert\leq A}\!\!
e^{i\xi(x-t)}e^{-\eps\lvert\xi\rvert^2}
\mathfrak{F}v(t,\xi)\lvert\xi\rvert^{\tfrac{n}{2k}}\,dtd\xi
\end{align*}
for certain constants $a$, $A$ and $B$ to be determined. 
Following Berhanu-Cordaro-Hounie \cite{MR2397326} we see that the first three integrals converge to holomorphic functions in a neighbourhood of the origin for $\eps\rightarrow 0$.

It remains to look at $I_1^\eps$. Suppose that $a$ is small enough such that
$B_a(0)\subseteq V$.
Let $\CC_j$, $1\leq j\leq N$ be open, acute cones such that
\begin{equation*}
\R^n=\bigcup_{j=1}^N\overline{\CC}_j
\end{equation*}
and the intersection $\overline{\CC}_j\cap\overline{\CC}_k$ has measure zero for $j\neq k$. Furthermore, let $\xi_0\in\CC_1$, 
$\CC_1\subseteq\Gamma$ and $\xi_0\notin\overline{\CC}_j$ for $j\neq 1$.
In particular that means that \eqref{M-FBIestimate2} holds on $B_a(0)\times\CC_1$, i.e.\
\begin{equation}\label{M-FBIestimate3}
\bigl\lvert\mathfrak{F}(\psi u)(x,\xi)\bigr\rvert\leq Ce^{-\omega_\M(\gamma\lvert\xi\rvert)}
\qquad x\in B_a(0),\; \xi\in \CC_1
\end{equation}

Furthermore for $j=2,\dotsc, N$ we can choose open cones $\Gamma_j$ with the property that $\xi_0\Gamma_j<0$ and
there is some positive constant $c_j$ such that
\begin{equation}\label{dualconeFBI}
\langle v,\xi\rangle\geq c_j\lvert v\rvert\cdot\lvert\xi\rvert\qquad \forall v\in\Gamma_j,\;\forall \xi\in\CC_j.
\end{equation}

We set
\begin{equation*}
f_j^\eps(x+iy)=\int_{\CC_j}\int_{B_a(0)}\!\!e^{i\xi(x+iy -t)-\eps\lvert\xi\rvert^2}
\mathfrak{F}v(t,\xi)\lvert\xi\rvert^{\tfrac{n}{2k}}\,dtd\xi
\end{equation*}
for $j\in\{2,\dots, N\}$. Note that each $f_j^\eps$ is entire if $\eps>0$ and for $\eps$ tending to $0$
the functions $f_j^\eps$ converge uniformly on compact subsets of the wedge $\R^m+i\Gamma_j$
to 
\begin{equation*}
f_j(x+iy)=\int_{\CC_j}\int_{B_a(0)}\!e^{i\xi(x+iy -t)}
\mathfrak{F}v(t,\xi)\lvert\xi\rvert^{\tfrac{n}{2k}}\,dtd\xi
\end{equation*}
which are also holomorphic on $\R^m\times i\Gamma_j$ due to \eqref{dualconeFBI}.

Similarly we define
\begin{align*}
f_1^\eps(x)&=\int_{\CC_1}\int_{B_a(0)}\!e^{i\xi(x -t)-\eps\lvert\xi\rvert^2}
\mathfrak{F}v(t,\xi)\lvert\xi\rvert^{\tfrac{n}{2k}}\,dtd\xi\\
\intertext{and}
f_1(x)&=\int_{\CC_1}\int_{B_a(0)}\!e^{i\xi(x -t)}\mathfrak{F}v(t,\xi)\lvert\xi\rvert^{\tfrac{n}{2k}}\,dtd\xi.
\end{align*}
The functions $f_1^\eps$, $\eps>0$,
 extend  to entire functions whereas $f_1$ is smooth due to \eqref{M-FBIestimate3} since $e^{-\omega_\M}$ is rapidly decreasing 
 (c.f.\ the remark after the proof of Lemma \ref{con}).
This decrease also shows that $f_1^\eps$ converges uniformly to $f_1$ in a neighbourhood of $0$ since
\begin{equation*}
\begin{split}
\bigl\lvert f_1(x)-f_1^\eps(x)\bigr\rvert&\leq\int_{\CC_1}\int_{B_a(0)}\!\bigl\lvert \mathfrak{F}v(t,\xi)\bigr\rvert 
\lvert\xi\rvert^{\tfrac{n}{2k}} \Bigl\lvert 1-e^{-\eps\lvert\xi\rvert^2}\Bigr\rvert\,dtd\xi\\
&\leq C\int_{\CC_1}\!\lvert\xi\rvert^{\tfrac{n}{2k}}e^{-\omega_\M(\gamma\lvert\xi\rvert)}\Bigl\lvert 1-e^{-\eps\lvert\xi\rvert^2}\Bigr\rvert\,d\xi
\end{split}
\end{equation*}
and the last integral converges to $0$ by the monotone convergence theorem.

In fact $f_1\in\E_\M$ because
\begin{equation*}
\begin{split}
\bigl\lvert D^\alpha f_1(x)\bigr\rvert 
&\leq\int_{\CC_1} \lvert\xi\rvert^{\tfrac{n}{2k}}
\bigl\lvert\xi^{\alpha}\mathfrak{F}v(t,\xi)\Bigr\rvert\,dtd\xi\\
&\leq C\int_{\CC_1}\lvert\xi\rvert^{\tfrac{n}{2k}+
\lvert\alpha\rvert} e^{-\omega_\M(\gamma \lvert\xi\rvert)}\,d\xi\\
&\leq C\int_{\CC_1}\lvert\xi\rvert^{\tfrac{n}{2k}-2n} 
\lvert\xi\rvert^{2n+\lvert\alpha\rvert} 
\tilde{h}_\M\biggl(\frac{1}{\gamma\lvert\xi\rvert}\biggr)\,d\xi\\
&\leq C\gamma^{-2n+\lvert\alpha\rvert}M_{2n+\lvert\alpha\rvert}
\int_{\CC_1}\lvert\xi\rvert^{\tfrac{n}{2k}-2n}\,d\xi\\
&\leq C\gamma^{\lvert\alpha\rvert}M_{\lvert\alpha\rvert},
\end{split}
\end{equation*}
where in the last step \eqref{derivclosed} is used.

So we have showed that on an open neighbourhood $U$ of the origin and some open cones $\Gamma_j$, 
$j=2,\dotsc, N$ that satisfy $\xi_0\Gamma_j<0$ we can write
\begin{equation*}
v\vert_U= v_0+\sum_{j=2}^N b_{\Gamma_j} f_j
\end{equation*}
with $v_0\in\E_\M(U)$ and $f_j$ holomorphic on $U +i\Gamma_j$ for $j=2,\dotsc,N$.
Hence $(0,\xi_0)\notin\WF_\M v$.
\end{proof}
We summarize our results regarding the description of $\WF_\M u$ in order to obtain 
the generalized Bony's Theorem alluded in the beginning of this section (c.f.\ Hoepfner-Medrado \cite{Hoepfner-MedradoPreprint}).
\begin{Thm}\label{BonyFBI}
Let $u\in\D^\prime(\Omega)$. For $(x_0,\xi_0)\in T^\ast\Omega\!\setminus\!\{0\}$ the following
statements are equivalent:
\begin{enumerate}
\item $(x_0,\xi_0)\notin \WF_\M u$
\item There are $U\in\U(x_0)$, open convex cones $\Gamma^j\subseteq\R^n$ with $\xi_0\Gamma^j <0$ and
 $\M$-almost analytic functions $F_j$ of slow growth in $U\times\Gamma^j_{\rho_j}$, $\rho_j>0$
and $j=1,\dotsc ,N$ for some $N\in\N$ such that 
\begin{equation*}
u\vert_{ U}=\sum_{j=1}^N b_{\Gamma^j}F_j.
\end{equation*}
\item There are $\varphi\in\D(\Omega)$ with $\varphi\equiv 1$ near $x_0$, $V\in\U(x_0)$ 
and an open cone $\Gamma$ containing $\xi_0$ such that \eqref{M-FBIestimate} holds.
\end{enumerate}
\end{Thm}
We can also give a local version of Theorem \ref{BonyFBI}.
\begin{Cor}
Let $u\in\D^\prime(\Omega)$ and $p\in\Omega$. Then the following is equivalent:
\begin{enumerate}
\item The distribution $u$ is of class $\E_\M$ near $p$.
\item There is a bounded sequence $(u_N)_N\subseteq\E^\prime(\Omega)$ and an open neighbourhood 
$V\subseteq \Omega$ of $p$
such that $u_N|_V=u|_V$ for all $N\in\N_0$ and \eqref{WF-M Estimate1} holds for $\Gamma =\R^n$ 
and some constant $Q>0$.
\item There exists an open neighbourhood $W\subseteq\Omega$ of $p$, $r>0$ and a smooth function
$F$ on $W+iB(0,r)$ such that $F|_W=u|_W$ and \eqref{Malmostest} holds for some constants $C,Q>0$.
\item There is a testfunction $\psi\in\D(\Omega)$ such that $\varphi_{\vert U}\equiv 1$ for 
some neighbourhood $U$ of $p$ and constants $C,\gamma>0$ such that 
\begin{equation*}
\sup_{(t,\xi)\in V\times \R^n}e^{\omega_\M(\gamma\lvert\xi\rvert)}
\bigl\lvert\mathfrak{F}(\psi u)(t,\xi) \bigr\rvert<\infty
\end{equation*}
for some $V\in\U(p)$.
\end{enumerate}
\end{Cor}
\begin{proof}
The equivalence of (1) and (2) is just Proposition \ref{MCharFT},
 whereas Corollary \ref{CharMalmostanalytic} shows that (1) implies (3).
 For the fact that (4) implies (1) we note that by Theorem \ref{M-FBIThm} we have that for all 
 $\xi\in\R^n\!\setminus\!\{0\}$ $(p,\xi)\notin\WF_\M u$. 
 Therefore $u$ has to be ultradifferentiable of class $\{\M\}$ near $p$.
Now we show that (4) follows from (3): Suppose that $u\in\E_\M(V)$ on a neighbourhood of $p$ 
and let $F\in\E(W+i\R^n)$ be an $\M$-almost analytic extension of $u$ 
on a relatively compact neighbourhood $W\subset\subset V$ of $p$.
We choose $\varphi\in\D(W)$, $0\leq\varphi\leq 1$ and $\varphi\equiv 1$ near $p$. We consider the map
\begin{equation*}
\theta:\; y\longmapsto \theta(y)=y-is\varphi(y)\frac{\xi}{\lvert\xi\rvert}.
\end{equation*}
for some $1>s>0$ to be determined.

Finally let $\psi\in\D(V)$ such that $\psi\equiv 1$ on $W$. 
As in the proof of Theorem \ref{M-FBIThm} we set $\psi(z)=\psi(x)$ for $z=x+iy\in\C^n$.
We put  $v=\psi F$ and consider the $n$-form 
\begin{equation*}
e^{Q(t,\xi ,z)}v(z)\,dz_1\wedge\dots\wedge dz_n
\end{equation*}
on 
\begin{equation*}
D_s=\biggl\{x+i\sigma\varphi(x)\frac{\xi}{\lvert\xi\rvert}\in\C^n\;\Big\vert\; 0<\sigma <s,\; x\in\supp v\biggr\}.
\end{equation*}
Stokes' Theorem gives us 
\begin{equation*}
\begin{split}
\mathfrak{F}v(t,\xi)
&=c_p\int_{\theta(\R^n)}\!e^{Q(t,\xi,z)} v(z,\bar{z})\,dz_1\wedge\dots\wedge dz_n\\
&+c_p\sum_{j=1}^n \int_{D_s}\!e^{Q(t,\xi,z)}\frac{\partial v}{\partial \bar{z}_j}(z,\bar{z})\,d\bar{z}_j\wedge dz_1\wedge\dots\wedge dz_n.
\end{split}
\end{equation*}
The second integral above is estimated in the same way as the last integral in \eqref{FBI-neccesity1}.
On the other hand the first integral on the right-hand side equals
\begin{equation*}
G(t,\xi)=c_p\int_{\R^n}\! e^{Q(t,\xi,\theta(y))}v(\theta(y))\det\theta^\prime(y)\,dy
\end{equation*}
We note that
\begin{equation*}
\real Q(t,\xi,\theta(y))\leq -s\varphi(y)\lvert\xi\rvert\bigl(1+O(s\varphi(y)\bigr)-c_0\lvert t-y\rvert^{2k}
\end{equation*}
and hence 
\begin{equation*}
\begin{split}
\lvert G(t,\xi)\rvert &\leq C\!\!\int_{B_\delta (p)}\!\!\! e^{\real Q(t,\xi,\theta(y))}\,dy
 +C\negmedspace\int_{\substack{\R^n\setminus B_\delta (p)\\ y\in\supp (v\circ\theta)}}\negthickspace\!\!\!
  e^{\real Q(t,\xi,\theta(y))}\,dy\\
&=I_1(t,\xi)+I_2(t,\xi),
\end{split}
\end{equation*}
where $B_\delta(p)\subseteq\{x\in\R^n\mid\varphi(x)=1\}$,
can be estimated as follows, c.f.\ Berhanu-Cordaro-Hounie \cite{MR2397326}. Set $s=\delta/4$.
We obtain
\begin{equation*}
I_1(t,\xi)\leq Ce^{-c\lvert\xi\rvert}
\end{equation*}
for all $\xi\in\R^n$ if $t$ is in some bounded neighbourhood of $p$.
Furthermore
\begin{equation*}
I_2(t,x)\leq C\int_{\substack{\R^n\setminus B_r(p)\\ y\in\supp (u\circ\theta)}}\negthickspace e^{-\lvert\xi\rvert\lvert t-y\rvert^{2k}}\,dy
\leq Ce^{-\bigl(\tfrac{\delta}{2}\bigr)^{2k}\lvert\xi\rvert}
\end{equation*}
for all $\xi$ and $\lvert t-p\rvert\leq\tfrac{\delta}{2}$.

Hence we have showed that there are constants $c,C>0$ such that
\begin{equation*}
\lvert\mathfrak{F}u(t,\xi)\rvert\leq Ce^{-\omega_\M(c\lvert\xi\rvert)}
\end{equation*}
for all $\xi\in\R^n$ and $t$ in a bounded neighbourhood of $p$.
\end{proof}
\section{Elliptic regularity}\label{sec:elliptic}
As mentioned in the introduction Albanese-Jornet-Oliaro \cite{MR2595651}  used the pattern of 
H\"ormander's proof of  \cite[Theorem 8.6.1]{MR1996773} (c.f.\ Remark \ref{HoeClassCom})
to prove elliptic regularity for operators with coefficients that are all in the same ultradifferentiable class
defined by a weight function, c.f.\ Remark \ref{Discussion}. 
Similarly H\"ormander's methods were applied by   Pilipovic-Teafanov-Tomic
 \cite{MR3463541},\cite{Pilipovic:2016oj} for certain classes 
that are defined by more degenerate sequences. 

It should be noted that the assumptions Albanese-Jornet-Oliaro put 
on the weight functions guarantee that the associated class is closed 
under composition and the inverse function theorem holds.
So it would be a reasonable conjecture that the regularity of the defining weight sequence is necessary for elliptic regularity to hold in the category of Denjoy-Carleman classes.
But there are weight functions obeying these conditions such that the associated function class cannot be described by regular weight sequences and on the other hand there are regular Denjoy-Carleman classes that cannot be defined by such weight functions,
see Bonet-Meise-Melikhov \cite{MR2387040}. 
It turns out, however, that the regularity of the weight sequence is not enough for
the proof of the elliptic regularity theorem, we also have to assume that \eqref{mg} holds. 
In that case the main result of Bonet-Meise-Melikhov \cite{MR2387040} implies that
the Denjoy-Carleman class can be described by a weight function that satisfies the conditions of Albanese-Jornet-Oliaro \cite{MR2595651}. 
Hence, we could use their elliptic regularity theorem, but we would have to show that their definition of the ultradifferentiable wavefront set coincides with the definition of H{\"o}rmander. 
Instead we give here a proof in full detail partially in preparation for the forthcoming paper 
F\"urd\"os-Nenning-Rainer-Schindl \cite{WeightMatrix}, where we deal with the problem 
in the far more general setting of the ultradifferentiable classes 
introduced in Rainer-Schindl \cite{MR3285413}, c.f.\ Remark \ref{Discussion}.

Furthermore, 
we  show here that H\"ormander's proof can be modified
in a way to investigate the regularity of solutions of a determined system of linear partial differential equations
\begin{align*}
P_{11}u_1+\dots +P_{1\nu}u_\nu&=f_1\\
\vdots\;\qquad \qquad &\;\:\vdots\\
P_{\nu 1}u_1+\dots +P_{\nu\nu}u_\nu &=f_\nu
\end{align*}
where $P_{j,k}$, $1\leq j,k\leq\nu$, is a partial differential operator with $\E_\M$-coefficients.

More precisely, using the geometric theory for the ultradifferentiable wavefront set developed in section \ref{sec:BV}, we can work in the following setting (see H{\"o}rmander \cite[chapter 6]{MR1996773} or Chazarain-Piriou \cite{MR678605}).

Let $M$ be an ultradifferentiable manifold of class $\{\M\}$ and $E$ and $F$ two vector bundles of class 
$\{\M\}$ on $M$ with the same fiber dimension $\nu$. 
An ultradifferentiable partial differential operator $P:\;\E_\M(M, E)\rightarrow\E_\M(M,F)$ of class $\{\M\}$
is given locally by
\begin{equation}\label{operatorrep}
Pu=
\begin{pmatrix}
P_{11} &\cdots & P_{1\nu}\\
\vdots & \ddots & \vdots\\
P_{\nu 1}&\cdots & P_{\nu\nu}
\end{pmatrix}
\begin{pmatrix}
u_1\\
\vdots\\
u_\nu
\end{pmatrix}
\end{equation}
where the $P_{jk}$ are linear partial differential operators with ultradifferentiable coefficients defined in suitable
chart neighbourhoods. If 
\begin{equation*}
Q(x,D)=\sum_{\lvert\alpha\rvert\leq m} q_\alpha(x)D^\alpha
\end{equation*}
 is a differential operator of order $\leq m$ on some open set $\Omega\subseteq\R^n$ then the principal symbol $q$
 is defined to be 
 \begin{equation*}
 q(x,\xi)=\sum_{\lvert\alpha\rvert=m}q_\alpha(x) \xi^\alpha.
 \end{equation*}
 Hence the order of $P$ is of order $\leq m$ iff no operator $P_{jk}$ on any chart neighbourhood is of 
 order higher than $m$ and $P$ is of order $m$ if the operator is not of order $\leq m-1$.
 The principal symbol $p$ of $P$ is an ultradifferentiable mapping defined on 
 $T^\ast M$ with values in
 the space of fiber-linear maps from $E$ to $F$ that is homogenous of degree $m$ in the fibers of 
 $T^\ast M$. It is given locally by
 \begin{equation}\label{symbolrep}
 p(x,\xi)=\begin{pmatrix}
 p_{11}(x,\xi)&\dots&p_{1\nu}(x,\xi)\\
 \vdots&\ddots &\vdots\\
 p_{\nu 1}(x,\xi)&\dots &p_{\nu\nu}(x,\xi)
 \end{pmatrix}
 \end{equation}
 where $p_{jk}$ is the principal symbol of the operator $P_{jk}$. 
 See Chazarain-Piriou \cite{MR678605} for more details.
 We say that $P$  is not characteristic (or non-characteristic) 
 at a point $(x,\xi)\in T^\ast M\!\setminus\!\{0\}$ if $p(x,\xi)$ 
 is an invertible  linear mapping. We define the set of all characteristic points
 \begin{equation*}
 \Char P=\{(x,\xi)\in T^{\ast}M\!\setminus\!\{0\}\,\colon P\text{ is characteristic at }(x,\xi)\}.
 \end{equation*}
\begin{Thm}\label{elliptic-regThm}
Suppose that $\M$ is a regular weight sequence that satifies also \eqref{mg}.
Let $M$ be an $\E_{\M}$-manifold and $E,F$ two ultradifferentiable vector bundles on $M$ of the same 
fiber dimension.
If $P(x,D)$ is a differential operator between $E$ and $F$ with $\E_\M$-coefficients
and $p$ its principal symbol, then
\begin{equation}\label{elliptic-regEq}
\WF_{\M}u\subseteq \WF_\M (Pu)\cup \mathrm{Char\,}P\qquad u\in\D^{\prime}(M,E).
\end{equation}
\end{Thm}
\begin{proof}
We write $f=Pu$.
Since the problem is local we work on some chart neighbourhood $\Omega$ such that 
in suitable trivializations of $E$ and $F$ we may write 
$u=(u_1,\dotsc,u_\nu)\in\D^\prime(\Omega,\C^\nu)$, 
$f=(f_1,\dotsc,f_\nu)\in\D^\prime(\Omega,\C^\nu)$
and $P$ and its principal symbol $p$ are of the form \eqref{operatorrep} and \eqref{symbolrep}, respectively.
In particular, $P$ is of order $m$ on $\Omega$.

We have to prove that if $(x_0,\xi_0)\notin \WF_\M f\cup \Char P$
then $(x_0,\xi_0)\notin\WF_\M u$.
Assuming this we find that there has to be a compact neighbourhood $K$ of $x_0$ and
a closed conic neighbourhood $V$ of $\xi_0$ in $\R^n\!\setminus\!\{0\}$ satisfying
\begin{align}
\det p(x,\xi)&\neq 0\qquad (x,\xi)\in K\times V\\
(K\times V)\cap \WF_\M (Pu)_j&=\emptyset\qquad j=1,\dots,\nu.
\end{align}

We consider the formal adjoint $Q=P^t$ of $P$ with respect of the pairing
\begin{equation*}
\langle f,g\rangle=\sum_{\tau=1}^\nu\int\!\! f_\tau (x)g_\tau (x)\,dx\qquad f,g\in\D(\Omega,\C^\nu).
\end{equation*}
If $P=(P_{jk})_{jk}$ then $Q=(Q_{jk})_{jk}=(P^t_{kj})_{jk}$ where $P_{jk}^t$ denotes the formal adjoint
of the scalar operator $P_{jk}(x,D)=\sum p^{\alpha}_{jk}(x)D^{\alpha}$, i.e.\ for $v\in\E(\Omega)$
\begin{equation*}
P^{t}_{jk}(x,D)v=\sum_{\lvert\alpha\rvert\leq m}(-D)^{\alpha}\Bigl(p^{\alpha}_{jk}(x)v(x)\Bigr).
\end{equation*}

Let $(\lambda_{N})_N\subseteq\D(K)$ be a sequence of test functions satisfying 
$\lambda_{N}\vert_{ U}\equiv 1$ on a fixed neighbourhood $U$ of $x_0$ for all $N$ and
for all $\alpha\in\N_0^n$ there are constants $C_{\alpha},h_{\alpha}>0$ such that
\begin{equation}\label{testfunctestimate}
\bigl\lvert D^{\alpha +\beta}\lambda_N\bigr\rvert\leq C_\alpha (h_\alpha N)^{\lvert\beta\rvert},\qquad \lvert\beta\rvert\leq N.
\end{equation}
If $u=(u^1,\dotsc ,u^\nu)\in\D^{\prime}(\Omega,\C^\nu)$, then we have that the sequence 
$u^{\tau}_N=\lambda_{2N}u^\tau$ is bounded in $\E^\prime$ and 
each of these distributions is equal to $u^\tau$ 
in $U$ for all $\tau$. 
Hence we have to prove that $(u^{\tau}_N)_N$ satisfies \eqref{WF-M Estimate1}, i.e.\
\begin{equation*}
\sup_{\substack{\xi \in V\\ N\in\N_0}}\frac{\lvert\xi\rvert^N \bigl\lvert\hat{u}_N^\tau\bigr\rvert}{Q^N M_N}<\infty
\end{equation*}
for a constant $Q>0$ independent of $N$.

In order to do so, set 
$\Lambda_N^\tau=\lambda_N e_\tau\in\D^\prime(\Omega,\C^\nu)$ and observe
\begin{equation*}
\hat{u}_N^\tau (\xi)=\bigl\langle u^\tau , e^{-i\langle\,.\,,\xi\rangle}\lambda_{2N}\bigr\rangle
=\bigl\langle u,e^{-i\langle\,.\,,\xi\rangle}\Lambda^\tau_{2N}\bigr\rangle.
\end{equation*}

Following the argument of H{\"o}rmander in the proof of \cite[Theorem 8.6.1]{MR1996773} we want to solve the equation
$Q g^\tau=e^{-ix\xi}\Lambda^{\tau}_{2N}$. We make the ansatz
\begin{equation*}
g^\tau=e^{-ix\xi}B(x,\xi)w^\tau
\end{equation*}
where $B(x,\xi)$ is the inverse matrix of the transpose of $p(x,\xi)$, 
which exists if $(x,\xi)\in K\times V$ and is homogeneous of degree $-m$ in $\xi$;
 note that the principal symbol of $Q=P^t$ is $B^{-1}(x,-\xi)$.
Using this we conclude that $w$ has to satisfy 
\begin{equation}\label{formal equation}
w^\tau-Rw^\tau=\Lambda^\tau_{2N}.
\end{equation}
Here $R=R_1+\dots +R_m$ with $R_j\lvert\xi\rvert^j$ being (matrix) differential operators of order $\leq j$ 
with coefficients in $\E_\M$ that are homogeneous of degree $0$ in $\xi$ if $x\in K$ and $\xi\in V$.

A formal solution of \eqref{formal equation} would be
\begin{equation*}
w^\tau=\sum_{k=0}^{\infty}R^k\Lambda^{\tau}_{2N}.
\end{equation*}
However, this sum may not converge and even if it would converge, in the estimates we want to obtain 
we are not allowed to consider derivatives of arbitrary high order. Hence we set
\begin{equation*}
w_N^\tau:=\sum_{j_1+\dots +j_k\leq N-m}R_{j_1}\cdots R_{j_k}\Lambda^\tau_{2N}
\end{equation*}
and compute
\begin{equation*}
w^\tau_N-Rw^\tau_N=\Lambda^{\tau}_{2N}-\sum_{\sum\limits_{s=1}^k j_s>N-m\geq\sum\limits_{s=2}^k j_s} R_{j_1}\dots R_{j_k}\Lambda^{\tau}_{2N}=\Lambda^{\tau}_{2N}-\rho^\tau_N.
\end{equation*}
Equivalently, we have
\begin{equation*}
Q\bigl(e^{-ix\xi}B(x,\xi)w_{N}^\tau\bigr)=e^{-ix\xi}\bigl(\Lambda^{\tau}_{2N}(x)-\rho^\tau_N(x,\xi)\bigr).
\end{equation*}
We obtain now
\begin{equation}\label{WF-M-EllipticRegEq1}
\begin{split}
\hat{u}^{\tau}_N(\xi) 
&=\bigl\langle u , e^{-i\langle\,.\,,\xi\rangle}\Lambda^{\tau}_{2N}\bigr\rangle \\
&=\bigl\langle u, Q\bigl(e^{-i\langle\,.\,,\xi\rangle} B(\,.\, , \xi) w^{\tau}_{N}\bigr)\bigr\rangle
 +\langle u ,e^{-i\langle\,.\,,\xi\rangle}\rho_{N}^{\tau}(\,.\,,\xi)\rangle\\
&=\langle f, e^{-i\langle\, .\,,\xi\rangle} B(\, .\, , \xi ) w^{\tau}_N\rangle
+\langle u,e^{-i\langle\, .\,,\xi\rangle} \rho_N^{\tau} (\, .\, ,\xi)\rangle
\end{split}
\end{equation}
and continue by estimating  the right-hand side of \eqref{WF-M-EllipticRegEq1}.
For this purpose we need the following Lemma.
\begin{Lem}\label{Lemma 1}
There exists constants $C$ and $h$ depending only on $R$ and the constants 
appearing in \eqref{testfunctestimate} such that,
 if $j=j_1+\dots +j_k$ and $j+\lvert\beta\rvert\leq 2N$,
we have
\begin{equation}\label{WF-M-RegEst1}
\left\lvert D^{\beta}\bigl(R_{j_1}\dots R_{j_k}\Lambda^{\tau}_{2N}\bigr)_\sigma\right\rvert\leq Ch^N M^{\tfrac{j+\lvert\beta\rvert}{N}}_{N}\lvert\xi\rvert^{-j}
\qquad \xi\in V,\sigma=1,\dotsc,\nu.
\end{equation}
\end{Lem}
\begin{proof}
Since both sides of \eqref{WF-M-RegEst1} are homogeneous of degree $-j$ in $\xi\in V$ 
it suffices to prove the lemma for $\lvert\xi\rvert=1$. 
Moreover we can write 
\begin{equation*}
\bigl(R_{j_1}\cdots R_{j_k}\Lambda^{\tau}_{2N}\bigr)_\sigma=\tilde{R}^{\tau}_{\sigma}\lambda_{2N}\qquad\quad 
\sigma=1,\dotsc,\nu
\end{equation*}
with $\tilde{R}^{\tau}_{\sigma}$ being a certain linear combination of products of components of the 
operators $R_{j_s}$.
Especially the coefficients of $\tilde{R}^{\tau}_{\sigma}$ are all of class $\{\M\}$ 
on a common neighbourhood of $K$ and since there are only finitely many of them we may assume that they all
can be considered as elements of $\E^q_\M(K)$ for some $q>0$. 
We denote the set of the coefficients of the operators $\tilde{R}^\tau_{2N}\sigma$ by 
$\mathcal{R}$.
Recall also from Remark \ref{HoeClassCom} that $\sqrt[N]{M_N}\rightarrow \infty$ and that
 there has to be a constant $\delta>0$ such that $N\leq \delta \sqrt[N]{M_N}$. 
Hence  \eqref{testfunctestimate} implies that for all $\alpha\in\N_0^n$ 
we have
\begin{equation}\label{testfunctestimate2}
\bigl\lvert D^{\alpha+\beta}\lambda_{2N}\bigr\rvert\leq C_\alpha h_{\alpha}^{\lvert\beta\rvert}(2N)^{\lvert\beta\rvert}
\leq C_\alpha (2h_{\alpha}\delta)^{\lvert\beta\rvert}
 M_{N}^{\tfrac{\lvert\beta\rvert}{N}}
\end{equation}
for $\lvert\beta\rvert\leq 2N$.

Considering all these arguments the proof of the lemma is a consequence of the following result.
\end{proof}
\begin{Lem}
Let $K\subseteq \Omega$ be compact, $(\lambda_N)_N\subseteq\D (K)$ a sequence satisfying
 \eqref{testfunctestimate2}, $q\geq 1$ and $a_1,\dotsc,a_{j-1}\in\mathcal{R}\cup\{1\}$. 
  Then there are constants $C,h>0$ independent of $N$ such that for $j\leq 2N$ we have
 \begin{equation}
\bigl\lvert D_{i_1}(a_1D_{i_2}(a_2\dots D_{i_{j-1}}(a_{j-1}D_{i_j}\lambda_{2N})\dots))\bigr\rvert\leq C h^j M_N^{\tfrac{j}{N}}.
 \end{equation}
\end{Lem}
\begin{proof}
We begin by noting  that \eqref{stlogconvex} implies that $m_jm_{k-j}\leq m_k$ for all $j\leq k\in\N$,
c.f.\ Komatsu \cite{MR0320743}.
Furthermore we can assume that there is a constant $C_1>1$ such that for all $k\leq j-1$
\begin{equation*}
\bigl\lvert D^\alpha a_{k}\bigr\rvert\leq
 C_1q^{\lvert\alpha\rvert}M_{\lvert\alpha\rvert}
\end{equation*}
on $K$.
Obviously the expression $ D_{i_1}a_1D_{i_2}a_2\dots D_{i_{j-1}}a_{j-1}D_{i_j}\lambda_{2N}$ 
can be written as a sum of terms of the form
$(D^{\alpha_1}a_1)\dotsb (D^{\alpha_{j-1}}a_{j-1})D^{\alpha_j}\lambda_{2N}$ where
$\lvert\alpha_1\rvert+\dots +\lvert\alpha_j\rvert=j$.

We set $h\geq C_1\max (q,h_0)$.
If there are $C_{k_1,\dots,k_j}$ terms with $\lvert\alpha_1\rvert =k_1,\dots,\lvert\alpha_j\rvert=k_j$ then 
we have the following estimate on $K$
\begin{equation*}
\begin{split}
\bigl\lvert D_{i_1}a_1D_{i_2}a_2\dots D_{i_{j-1}}a_{j-1}D_{i_j}\lambda_{2N}\bigr\rvert
&\leq C\sum q^{j-k_j}C_1^{j-1}C_{k_1,\dots,k_j}m_{k_1}\cdots m_{k_{j-1}}k_1!\cdots k_{j-1}!h_0^{k_j}M_N^{\tfrac{k_j}{N}}\\
&\leq C h^j\sum m_{j-k_j}C_{k_1,\dotsc,k_j}k_1!\cdots k_{j-1}! M_N^{\tfrac{k_j}{N}}\\
&\leq C h^j\sum C_{k_1,\dotsc,k_j}\frac{k_1!\cdots k_{j-1}!}{(j-k_j)!}M_{j-k_j}M_N^{\tfrac{k_j}{N}}.
\end{split}
\end{equation*}
Since $j-k_j\leq 2N$, we observe that \eqref{mg} implies that there are two indices $\sigma_1,\sigma_2\leq N$, $\sigma_1+\sigma_2=k-k_j$ such that $M_{k-k_j}\leq C\rho^{j-k_j}M_{\sigma_1}M_{\sigma_2}$ for some constants $C,\rho$ that are independent of $j$ and $N$.
Now we have
\begin{equation*}
M_{j-k_j}M_N^{\tfrac{k_j}{N}}=C\rho^{j-k_j}M_{\sigma_1}M_{\sigma_2}M_N^{\tfrac{k_j}{N}}
\leq C\rho^{j-k_j}M_N^{\tfrac{\sigma_1+\sigma_2}{N}}M_N^{\tfrac{k_j}{N}}=C\rho^{j-k_j}M_N^{\tfrac{j}{N}}.
\end{equation*}
since $\sqrt[N]{M_N}$ is increasing.
As noted in \cite{MR2595651} it is possible to estimate
\begin{equation*}
\frac{k_1!\cdots k_{j-1}!}{(j-k_j)!}=\frac{k_1!\cdots k_{j-1}!k_j!j!}{(j-k_j)!k_j!j!}\leq 2^j\frac{k_1!\cdots k_j!}{j!},
\end{equation*}
and also (c.f.\ \cite[p. 308]{MR1996773})
\begin{equation*}
\sum C_{k_1,\dotsc,k_j}k_1!\cdots k_j!=(2j-1)!!.
\end{equation*}

Since $\tfrac{(2j-1)!!}{j!2^j}\leq 1$ we obtain
\begin{equation*}
\begin{split}
\bigl\lvert D_{i_1}a_1D_{i_2}a_2\dots D_{i_{j-1}}a_{j-1}D_{i_j}\lambda_{2N}\bigr\rvert
&\leq C (4\rho h)^j\frac{(2k-1)!!}{j!2^j}M_N^{\tfrac{j}{N}}\\
&\leq C (4\rho h)^jM_N^{\tfrac{j}{N}}.
\end{split}
\end{equation*}
\end{proof}
In order to estimate $\hat{u}^\tau_N$, we note that due to the boundedness of the sequence 
$(u^\tau_N)_N\subseteq\E^{\prime}$ the Banach-Steinhaus theorem implies 
that there are constants $\mu$ and $c$ such that
\begin{equation*}
\lvert\hat{u}_N^\tau\rvert\leq c\bigl(1+\lvert\xi\rvert\bigr)^\mu
\end{equation*}
for all $N$ and therefore if $\lvert\xi\rvert\leq \sqrt[N]{M_{N}}$ then
\begin{equation}\label{EllipticREGULAR0}
\bigl\lvert\xi\rvert^N\lvert\hat{u}_N^\tau\rvert\leq CM_{N}^{\tfrac{N+\mu}{N}}
\leq C\delta^{\mu N}M_N,
\end{equation}
since \eqref{derivclosed} implies that there is a constant $\delta>0$ such that
$\sqrt[N]{M_N}\leq \delta\sqrt[N-1]{M_{N-1}}$.

Hence it suffices to estimate the terms on the right-hand side of \eqref{WF-M-EllipticRegEq1} for $\xi\in V$, $\lvert\xi\rvert>\sqrt[N]{M_{N}}$.
We begin with the second term.

As in the scalar case there are constants $\mu$ and $C>0$ that only depend on $u$ and $K$ such that for all
$\psi\in\D(\Omega, \C^\nu)$ with $\supp \psi\subseteq K$
\begin{equation*}
\bigl\lvert\langle u,\psi\rangle\bigr\rvert
\leq C\sum_{\lvert\alpha\rvert\leq \mu}\sup_K\bigl\lvert D^{\alpha}\psi\bigr\rvert.
\end{equation*}
Note that $\supp_x\rho_N^\tau(\,.\,,\xi)\subseteq K$ for all $\xi\in V$ and $N\in\N$. Thence
 \begin{equation*}
 \begin{split}
\bigl\lvert \langle u,e^{-i\langle\, .\,,\xi\rangle} \rho_N^{\tau} (\, .\, ,\xi)\rangle\bigr\rvert 
&\leq C \sum_{\lvert\alpha\rvert \leq \mu} \sum_{\beta\leq\alpha}
 \lvert\xi\rvert^{\lvert\alpha\rvert-\lvert\beta\rvert} 
 \sup_{x\in K}\bigl\lvert D^{\beta}_x\rho^\tau_N(x,\xi)\bigr\rvert\\
&\leq C\sum_{\lvert\alpha\rvert\leq \mu}
\lvert\xi\rvert
^{\mu-\lvert\alpha\rvert} \sup_{x\in K}\bigl\lvert D_x^{\alpha}\rho^{\tau}_N(x,\xi)\bigr\rvert
 \end{split}
 \end{equation*}
 for $\xi\in V$, $\lvert\xi\rvert\geq 1$ and $N\in \N$.
 There are at most $2^N$ terms of the form $R_{j_1}\dots R_{j_k}\Lambda^{\tau}_{2N}$
in  $\rho_N^\tau$ and each term can be estimated by
 \eqref{WF-M-RegEst1} setting $N\geq j>N-m$ and hence
 \begin{equation*}
 \bigl\lvert D_x^{\alpha}\rho^\tau_N(x,\xi)\bigr\rvert\leq Ch^{N}2^N 
 \lvert \xi\rvert^{m-N}M_{N}^{\tfrac{N+\lvert\alpha\rvert}{N}}
 \end{equation*}
 for $x\in K$ and $\xi\in V$, $\lvert\xi\rvert>1$.
 Applying \eqref{derivclosed} therefore gives
 \begin{equation}\label{EllipticREGULAR1}
\begin{split}
 \bigl\lvert \langle u,e^{-i\langle\, .\,,\xi\rangle} \rho_N^{\tau} (\, .\, ,\xi)\rangle\bigr\rvert
 &\leq C h^N 2^{N}\lvert\xi\rvert^{\mu+m-N}M_{N}^{\tfrac{N+\mu}{N}}\\
&\leq C h^N \lvert\xi\rvert^{\mu+m-N}M_N.
\end{split}
 \end{equation}
 
The first term in \eqref{WF-M-EllipticRegEq1} is more difficult to estimate.
To begin with, observe that Lemma \ref{Lemma 1} gives 
 \begin{equation*}
 \begin{split}
\bigl\lvert D^{\beta}w^\tau_N(x,\xi)\bigr\rvert
&\leq C h^N\sum_{j=0}^{N-m}  M_{N}^{\tfrac{j+\lvert\beta\rvert}{N}}\lvert\xi\rvert^{-j}\\
&\leq Ch^N M_{N}^{\tfrac{\lvert\beta\rvert}{N}}\sum_{j=0}^{N-m}M_{N}^{\tfrac{j-j}{N}}\\
&\leq C h^N M_{N}^{\tfrac{\lvert\beta\rvert}{N}}(N-m)\\
&\leq C h^N M_{N}^{\tfrac{\lvert\beta\rvert}{N}}
\end{split}
 \end{equation*}
for $N>m$, $\lvert\beta\rvert\leq N$ and $\xi\in V$, $\lvert\xi\rvert > \sqrt[N]{M_{N}}$.
Recall that for $N\leq m$ we have set $w^\tau_N=\Lambda^\tau_{2N}=\lambda^\tau_{2N}e_\tau$.
Hence by the above and \eqref{testfunctestimate2} it follows that
\begin{equation}\label{WF-M-ell-Auxest1}
\bigl\lvert D^{\beta}w^\tau_N(x,\xi)\bigr\rvert\leq C h^NM_{N}^{\tfrac{\lvert\beta\rvert}{N}}
\end{equation}
for all $N\in\N$, $\lvert\beta\rvert\leq N$ and $\xi\in V$, $\lvert\xi\rvert>\sqrt[N]{M_{N}}$.

On the other hand, since the components of $B(x,\xi)$ are ultradifferentiable of class $\{\M\}$ 
and homogeneous in $\xi\in V$
 of degree $- m$  we note that it is possible to show similarly to above, 
 using an analogue to Lemma \ref{Lemma 1},
the following estimate
\begin{equation}\label{WF-M-ell-Auxest2}
 \bigl\lvert D^{\beta}_x\bigl(w_N^\tau(x,\xi)\lvert\xi\rvert^{m}B(x,\xi)\bigr)\bigr\rvert
 \leq C h^N M_{N}^{\tfrac{\lvert \beta\rvert}{N}} \qquad \lvert\beta\rvert\leq N,\,\xi\in V,\,\lvert\xi\rvert > \sqrt[N]{M_{N}}.
 \end{equation}
 
In order to finish the proof of Theorem \ref{elliptic-regThm} we need an additional Lemma.
 \begin{Lem}
Let $f\in\D^{\prime}(\Omega)$, $K$ be a compact subset of $\Omega$ 
and $V\subseteq\R^n\!\setminus\!\{0\}$ a closed cone such that
\begin{equation*}
\WF_\M f\cap (K\times V)=\emptyset.
\end{equation*}
Furthermore let $w_N\in\D (\Omega\times V)$ such that $\supp w_N\subseteq K\times V$ and  \eqref{WF-M-ell-Auxest1} holds.

If $\mu$ denotes the order of $f$ in a neighbourhood of $K$ then
\begin{equation}\label{WF-M-ell-Auxest 3}
\Bigl\lvert \widehat{w_Nf}(\xi)\Bigr\rvert
=\bigl\lvert\bigl\langle w_N(\;.\;,\xi)f,e^{-i\langle \,.\,,\xi\rangle}\bigr\rangle\bigr\rvert
\leq C h^N\lvert\xi\rvert^{\mu +n-N}M_{N-\mu-n},
\end{equation}
for $N>\mu+n$ and $\xi\in \Gamma$, $\lvert\xi\rvert>\sqrt[N]{M_{N}}$.
 \end{Lem}
 \begin{proof}
 By Proposition \ref{WF-M Charakterisierung} we can find a sequence $(f_N)_N$ that is bounded in 
 $\E^{\prime ,\mu}$ and equal to $f$ in some neighbourhood of $K$ and 
 \begin{equation}\label{ellipticHelp1}
 \left\lvert\hat{f}_N(\eta)\right\rvert\leq C \frac{Q^NM_N}{\lvert \eta\rvert^N}\qquad\eta\in W
 \end{equation}
 where $W$ is a conic neighbourhood of $\Gamma$. Then $w_Nf=w_Nf_{N^\prime}$ for
 $N^{\prime}=N-\mu -n$. 
 
If we denote the partial Fourier transform of $w_N(x,\xi)$ by
\begin{equation*}
\hat{w}_N(\eta,\xi)=\int_\Omega \! e^{-ix\eta}w_N(x,\xi)\,dx
\end{equation*}
then obviously \eqref{WF-M-ell-Auxest1} is equivalent to
\begin{equation*}
\bigl\lvert\eta^\beta\hat{w}_N(\eta,\xi)\bigr\rvert\leq C h^N M_{N}^{\tfrac{\lvert\beta\rvert}{N}}
\end{equation*}
for $\lvert\beta\rvert\leq N$, $\xi\in V$, $\lvert\xi\rvert>\sqrt[N]{M_{N}}$ and $\eta\in\R^n$.
Since $\lvert\eta\rvert\leq \sqrt{n}\max\lvert\eta_j\rvert$ we conclude that
\begin{equation}\label{EllipticHelp1a}
\lvert\eta\rvert^\ell\lvert\hat{w}_N(\eta,\xi)\rvert\leq C h^N M^{\tfrac{\ell}{N}}_{N}
\end{equation}
for $\ell\leq N$, $\eta\in\R^n$ and $\xi\in V$, $\lvert\xi\rvert>\sqrt[N]{M_{N}}$.
Hence we obtain \begin{equation}\label{ellipticHelp2}
 \begin{split}
 \Biggl(\lvert\eta\rvert+M_{N}^{\tfrac{1}{N}}\Biggr)^N\bigl\lvert\hat{w}_N(\eta,\xi)\bigr\rvert
 & =\sum_{k=0}^N\binom{N}{k}M_{N}^{\tfrac{k}{N}}\lvert \eta\rvert^{N-k}\lvert\hat{w}_N(\eta,\xi)\bigr\rvert\\
 &\leq Ch^N\sum_{k=0}^N\binom{N}{k}M_{N}^{\tfrac{k}{N}}M_{N}^{\tfrac{N-k}{N}}\\
& \leq C h^N M_{N}
 \end{split}
 \end{equation}
 if $\eta\in\R^n$, $\xi\in V$ and $\lvert\xi\rvert>\sqrt[N]{M_{N}}$. 
 Like H{\"o}rmander \cite{MR1996773} and Albanese-Jornet-Oliaro \cite{MR2595651} we consider 
 \begin{align*}
 \widehat{w_Nf}(\xi)&=\frac{1}{(2\pi)^n}\int\! \hat{w}_N(\eta,\xi)\hat{f}_{N^\prime}(\xi-\eta)\,d\eta \\
 &=\frac{1}{(2\pi)^n}\int\limits_{\lvert\eta\rvert<c\lvert\xi\rvert}\! \hat{w}_N(\eta,\xi)\hat{f}_{N^\prime}(\xi-\eta)\,d\eta  
 +\frac{1}{(2\pi)^n}\int\limits_{\lvert\eta\rvert>c\lvert\xi\rvert}\! \hat{w}_N(\eta,\xi)\hat{f}_{N^\prime}(\xi-\eta)\,d\eta 
 \end{align*}
for some $0<c<1$. The boundedness of the sequence $(f_N)_N$ in $\E^{\prime,\mu}$ implies as before that
\begin{equation*}
\bigl\lvert \hat{f}_N(\xi)\bigr\rvert\leq C \bigl(1+\lvert\xi\rvert\bigr)^\mu.
\end{equation*}
Hence we conclude that
 \begin{equation*}
(2\pi)^n\Bigl\lvert \widehat{w_Nf}(\xi)\Bigr\rvert \leq \bigl\lVert \hat{w}_N(\,.\,,\xi)\bigr\rVert_{L^1}
\sup_{\lvert\xi-\eta\rvert<c\lvert\xi\rvert}\bigl\lvert \hat{f}_{N^\prime}(\eta)\bigr\rvert +
C\int\limits_{\lvert\eta\rvert>c\lvert\xi\rvert}\!
\bigl\lvert \hat{w}_{N}(\eta,\xi)\bigr\rvert\bigl(1+c^{-1}\bigr)^\mu \bigl(1+\lvert\eta\rvert)^\mu\,d\eta
 \end{equation*}
 since $\lvert\eta\rvert\geq c\lvert\xi\rvert$ gives $\lvert\xi +\eta\rvert\leq (1-c^{-1})\lvert\eta\rvert$.

On the other hand there is  a constant $0<c<1$ such that $\eta\in W$ when $\xi\in V$
and $\lvert\xi-\eta\rvert\leq c\lvert\xi\rvert$.  Then $\lvert\eta\rvert\geq (1-c)\lvert\xi\rvert$ and
we can replace the supremum above by
$\sup_{\eta\in W}\lvert\hat{f}_{N^\prime}(\eta)\rvert$.
Furthermore by \eqref{ellipticHelp2}
\begin{equation*}
\begin{split}
\bigl\lVert \hat{w}_N(\;.\;,\xi)\bigr\rVert_{L_1}
&=\int_{\R^n}\bigl\lvert\hat{w}_N(\eta,\xi)\bigr\rvert\,d\eta\\
&\leq Ch^NM_{2N}^{\tfrac{N}{2N}}\int_{\R^n}\Bigl(\lvert\eta\rvert+\sqrt[2N]{M_{2N}}\Bigr)^{-N}\,d\eta\\
&\leq Ch^NM_{N}\int_{\sqrt[2N]{M_{2N}}}^\infty\!\!\! s^{-N^\prime-1}\,ds\\
&\leq Ch^NM_{N}\frac{M_{N}^{-\tfrac{N^\prime}{N}}}{N^\prime}\\
&\leq Ch^NM_{N}^{\tfrac{\mu +n}{N}}.
\end{split}
\end{equation*}
Thence it follows for $\xi\in V$, $\lvert\xi\rvert>\sqrt[N]{M_{N}}$, that
\begin{equation*}
\begin{split}
\Bigl\lvert\widehat{w_Nf}(\xi)\Bigr\rvert 
&\leq C_1 (1-c)^{-N^\prime}\bigl\lVert \hat{w}_N(\,.\,,\xi)\bigr\rVert_{L^1}
\lvert\xi\rvert^{-N^{\prime}}\sup_{\eta\in W}\bigl\lvert\hat{f}_{N^\prime}(\eta)\bigr\rvert\lvert\eta\rvert^{N^{\prime}}\\
&\quad+C_2\bigl(1+c^{-1}\bigr)^{N^{\prime}+\mu}\int_{\lvert\eta\rvert\geq c\lvert\xi\rvert} 
\negthickspace(1+\lvert\eta\rvert)^\mu \lvert \hat{w}_N(\eta,\xi)\rvert\,d\eta\\
&\leq C_1 h^N M_{N}^{\tfrac{n+\mu}{N}} Q^{N^\prime}M_{N^\prime}\lvert\xi\rvert^{-N^{\prime}}
+C_2\tilde{h}^{N} M_{N} \int_{\lvert\eta\rvert>c\lvert\xi\rvert}\negthickspace
\lvert\eta\rvert^{-N^\prime-n}\,d\eta\\
&\leq C h^{N} M_{N^\prime}\lvert\xi\rvert^{-N^{\prime}}
\end{split}
\end{equation*}
where we have also used \eqref{derivclosed}, \eqref{ellipticHelp1} and \eqref{EllipticHelp1a}.
 \end{proof}
 Due to \eqref{WF-M-ell-Auxest2} we can replace $w_N$ in \eqref{WF-M-ell-Auxest 3} with 
 $(w_N^\tau\lvert\xi\rvert^m B)_\sigma$, $\sigma=1,\dotsc,\nu$, and obtain
 \begin{equation}\label{EllipticREGULAR2}
 \bigl\lvert\bigl\langle f, e^{-i\langle\,.\,,\xi\rangle} B(\,.\,,\xi)w_N^\tau\bigr\rangle\bigr\rvert
 \leq Ch^N\lvert \xi\rvert^{\mu +n-N}M_{N-\mu-n}
 \end{equation}
 for $\xi\in V$, $\lvert \xi\rvert>\sqrt[N]{M_{N}}$. 
 
 We consider now the sequence $(v^\tau_N)_N=(u^\tau_{N+m+n+\mu})_N$.
 If $\xi\in V$, $\lvert\xi\rvert\leq \sqrt[N]{M_{N}}$, 
  then by \eqref{EllipticREGULAR0}
 \begin{equation*}
 \lvert\xi\rvert^N\bigl\lvert\hat{v}^\tau_N\rvert\leq C\delta^NM_N.
 \end{equation*}
 
 On the other hand \eqref{WF-M-EllipticRegEq1}, \eqref{EllipticREGULAR1} and \eqref{EllipticREGULAR2} give
 \begin{equation*}
 \begin{split}
 \lvert\xi\rvert^N\bigl\lvert\hat{v}^\tau_N(\xi)\bigr\rvert
 &\leq C_1 h_1^N M_{N+m}\lvert\xi\rvert^{-m}+C_2h_2^N M_{N+\mu +m+n}\lvert\xi\rvert^{-n}\\
 &\leq Ch^N M_N
 \end{split}
 \end{equation*}
 for $\xi\in V$, $\lvert\xi\rvert>\sqrt[N]{M_{N}}$.
 
 Therefore we have shown for all $\tau=1,\dotsc,\nu$ that the bounded sequence 
 $(v^\tau_N)_N\subseteq\E^\prime(\Omega)$ satisfies
 \begin{equation*}
 \sup_{\substack{\xi\in V\\ N\in\N}}
 \frac{\lvert\xi\rvert^N\bigl\lvert v^\tau_N(\xi)\bigr\rvert}{Q^NM_N}<\infty
 \end{equation*}
 for some $Q>0$. Clearly $u^\tau\vert_U\equiv (v^\tau_N)\vert_U$ and hence
 \begin{equation*}
 (x_0,\xi_0)\notin\WF_{\M} u^\tau
 \end{equation*}
 for all $\tau=1,\dotsc,\nu$.
\end{proof}
For elliptic operators, i.e.\ operators $P$ with $\Char P=\emptyset$, the following holds obviously.
\begin{Cor}
If $P$ is an elliptic operator with ultradifferentiable coefficients of class $\{\M\}$  and $u\in\D^\prime$ then
\begin{equation*}
\WF_\M Pu=\WF_\M u.
\end{equation*}
\end{Cor}
\section{Uniqueness Theorems}\label{sec:uniq}
H{\"o}rmander \cite{MR0294849} and Kawai (see \cite{MR0420735}) independently noticed that results 
like Theorem \ref{elliptic-regThm} in the analytic category
can be used to prove Holgrem's Uniqueness Theorem \cite{zbMATH02662678}.
We show here that Theorem \ref{elliptic-regThm} can also be used to give a quasianalytic version
of Holgrem's Uniqueness Theorem. We follow mainly the presentation of H{\"o}rmander \cite{MR1996773}.

First recall H\"ormander \cite[Theorem 6.1.]{MR1275197}:
\begin{Prop}\label{sec:Uniq-prop1}
Let $I\subseteq\R$ be an interval and $x_0\in\partial\supp u$ then 
$(x_0,\pm 1)\in\WF_\M u$ for any quasianalytic regular weight sequence $\M$.
\end{Prop}
As H\"ormander noted in \cite{MR1275197} Proposition \ref{sec:Uniq-prop1} 
immediately generalizes to a result for several variables 
(c.f.\ \cite{MR1996773}[Theorem 8.5.6], see Kim-Chung-Kim \cite{MR1826922} for a similar result):
\begin{Thm}\label{sec:Uniq-thm1}
Let $\M$ be a quasianalytic regular weight sequence, $u\in\D^\prime(\Omega)$, $x_0\in\supp u$ 
and $f:\Omega \rightarrow \R$ a 
function of class $\{\M\}$ with the following properties:
\begin{equation*}
df(x_0)\neq 0,\quad f(x)\leq f(x_0)\quad\text{if }x_0\neq x\in\supp u
\end{equation*}
Then we have
\begin{equation*}
(x_0,\pm df(x_0))\in\WF_\M u.
\end{equation*}
\end{Thm}
\begin{proof}
If we replace $f$ by $f(x)-\lvert x-x_0\rvert^2$ we see that we may assume that 
$f(x)<f(x_0)$ for $x_0\neq x\in\supp u$. 
Furthermore, since $d f(x_0)\neq 0$ we can assume that $x_0=0$ and $f(x)=x_n$.
Next we choose a neighbourhood $U$ of $0$ in $\R^{n-1}$ so that $U\times\{ 0\}\subset\subset \Omega$.
By assumption $\supp u\cap (\bar{U}\times \{0\})=\{0\}$. Hence there is an open interval 
$I\subseteq\R$ with $0\in I$ such that
\begin{equation}\label{Supp}
U\times I\subset\subset\Omega \quad\text{\&}\quad \supp u\cap\bigl(\partial U\times I\bigr)=\emptyset.
\end{equation}
If $A$ is an entire analytic function in the variables $x^\prime=(x_1,\dotsc,x_{n-1})$ then we consider  the distribution $U_A\in\D^\prime(I)$ given by $\langle U_A,\psi\rangle=\langle u_A\otimes\psi\rangle$. Note $U_A$ is well defined due to
\eqref{Supp}.
By \cite[Theorem 8.5.4']{MR1996773} we have that
\begin{equation*}
\WF_\M\bigl(U_A\bigr)\subseteq\bigl\{(x_n,\xi_n)\in I\times\R\!\setminus\!\{0\}\mid
\exists x^\prime\in U:\;
(x^\prime,x_n,0,\xi_n)\in\WF_\M u\bigr\}.
\end{equation*}
Note that $(x^\prime,x_n)$ above must be close to $0$ for $x_n$ small. 

Assume, e.g., that $(0,e_n)\notin\WF_\M u$, $e_n=(0,\dotsc,0,1)$. 
Then $I$ can be chosen so small that $(x,e_n)\notin\WF_\M u$ for $x\in U\times I$.
We conclude that $(x_n,1)\notin\WF_\M U_A$ if $x_n\in I$. Proposition \ref{sec:Uniq-prop1} 
implies that $U_A=0$ on $I$ since $U_A=0$ on $I\cap\{x_n>0\}$.
That means actually that
\begin{equation*}
\Bigl\langle u\vert_{U\times I},A\otimes \varphi\Bigr\rangle =0
\end{equation*}
for all $\varphi\in\D(I)$. 
Since $A$ was chosen arbitrarily from a dense subset of $\E(\R^{n-1})$ it follows that $u=0$
on $U\times I$.
\end{proof}
For the rest of this section $\M$ is going to be a quasianalytic regular weight sequence that satisfies \eqref{mg}.

In order to give Theorem \ref{sec:Uniq-thm1} a more invariant form we need to recall
some facts from \cite{MR1996773}.
\begin{Def}
Let $F$ be a closed subset of a $\CC^2$ manifold $X$. The \emph{exterior normal set} 
$N_e(F)\subseteq T^\ast X\!\setminus\!\{0\}$ is defined as the set of all points $(x_0,\xi_0)$ such that 
$x_0\in F$ and there exists a real valued function $f\in\CC^2(X)$ with $df (x_0)=\xi_0\neq 0$ and
$f(x)\leq f(x_0)$ when $x\in F$.
\end{Def}
In fact, following the remarks in \cite[p.\ 300]{MR1996773} we observe that it would be sufficient for $f$ to be 
defined locally around $x_0$. Furthermore $f$ could then also be chosen real-analytic in a chart neighbourhood 
near $x_0$.
If $g$ is $\CC^1$ near a point $\tilde{x}\in F$ and $dg(\tilde{x})=\tilde{\xi}\neq 0$ then 
$(\tilde{x},\tilde{\xi})\in \overline{ N_e(F)}\subseteq T^\ast X\!\setminus\!\{0\}$. 
It is clear that if $(x_0,\xi_0)\in N_e(F)$ then $x_0\in\partial F$. 
In fact, if $\pi: T^\ast\Omega \rightarrow \Omega$ is the canonical projection then
$\pi(N_e(F))$ is dense in $\partial F$, see \cite[Proposition 8.5.8.]{MR1996773}.
The \emph{interior normal set} $N_i(F)\subseteq T^\ast X\!\setminus\!\{0\}$ consists of all points $(x_0,\xi_0)$
with $(x_0,-\xi_0)\in N_e(F)$. The \emph{normal set} of $F$ is defined as 
$N(F)=N_e(F)\cup N_i(F)\subseteq T^\ast X\!\setminus\!\{0\}$.

In this notation Theorem \ref{sec:Uniq-thm1} takes the following form.
\begin{Thm}\label{sec:Uniq-thm2}
Let  $u\in\D^\prime(\Omega)$. Then
\begin{equation*}
\overline{N(\supp u)}\subseteq \WF_\M u
\end{equation*}
\end{Thm}
Theorem \ref{sec:Uniq-thm2} combined with Theorem \ref{elliptic-regThm} gives
\begin{Thm}
Let $P$ be a partial differential operator with $\E_\M$-coefficients and 
$u\in\D^\prime(\Omega)$ a solution of $Pu=0$.
Then 
\begin{equation*}
\overline{N(\supp u)}\subseteq \Char P,
\end{equation*}
i.e., the principal symbol $p_m$ of $P$ must vanish on $N(\supp u)$.
\end{Thm}
In fact, we can now derive the \emph{quasianalytic Holgrem Uniqueness Theorem}.
We recall that a $\mathcal{C}^1$-hypersurface $M$ is characteristic at a point $x$ 
with respect to a partial differential operator $P$,
iff for a defining function $\varphi$ of $M$ near $x$ we have that $(x,d\varphi(x))\in\Char P$.
\begin{Cor}\label{QuasiHolmgren}
Let $P$ a partial differential operator with $\E_\M$-coefficients.
If $X$ is a $\CC^1$-hypersurface in $\Omega$ that is non-characteristic at $x_0$ and 
$u\in\D^\prime(\Omega)$ a solution of $Pu=0$ that vanishes on one side of $X$ near $x_0$ then
$u\equiv 0$ in a full neighbourhood of $x_0$.
\end{Cor}
In fact, (c.f.\ Zachmanoglou \cite{MR0240442}) it is possible to reformulate Corollary \ref{QuasiHolmgren}
\begin{Cor}\label{QuasiHolmgren2}
Let $P$ a differential operator with coefficients in $\E_\M(\Omega)$.
Furthermore let $F\in\E_\M(\R^n)$ be a real-valued function of the form
\begin{equation*}
F(x)=f\bigl(x^\prime\bigr)-x_n,\qquad x^\prime =(x_1,\dotsc,x_{n-1})
\end{equation*}
where $f\in\E_\M(\R^{n-1})$ and suppose that the level hypersurfaces of $F$ are nowhere characteristic with
respect to $P$ in $\Omega$. Set also $\Omega_c=\{x\in\Omega\mid F(x)<c\}$ for $c\in\R$.
If $u\in\D^\prime(\Omega)$ is a solution of $P(x,D)u=0$ and there is $c\in\R$ such that
$\Omega_c\cap\supp u$ is relatively compact in $\Omega$, then $u=0$ in $\Omega_c$.
\end{Cor}
\begin{proof}
We set for $c\in\R$
\begin{equation*}
\omega_c=\bigl\{x\in\Omega\mid F(x)=c\bigr\}.
\end{equation*}
Note that for each $c\in\R$ the set $\omega_c$ is not relatively compact in $\Omega$. 
Therefore also $\Omega_c$ is not relatively compact in $\Omega$ for any $c$ since $\partial\Omega_c=\omega_c$.

By assumption there is a $c\in\R$ such that $K=\supp u\cap \overline{\Omega}_c$ is compact in $\Omega$.
In particular, $K$ is bounded in $\Omega$. Hence there has to be $\tilde{c}<c$ such that 
\begin{equation*}
K\subseteq\big\{x\in\Omega\mid \tilde{c}\leq F(x)\leq c\bigr\}.
\end{equation*}
Let $c_1<c$ be the greatest real number such that the inclusion above holds for $\tilde{c}=c_1$.
Since $K$ is compact there is a point $p\in\partial K$ such that $F(p)=c_1$.
It follows that $p\in\partial\supp u\cap\omega_{c_1}$. 
Thus we can apply Corollary \ref{QuasiHolmgren} because $\omega_{c_1}$ is nowhere characteristic for $P$.
Hence $u$ vanishes in a full neighbourhood of $p$. This contradicts the choice of $c_1$.
We conclude that $u$ has to vanish on $\Omega_c$.
\end{proof}
Note that in \cite{MR0404822} H\"ormander used the analytic version of Corollary \ref{QuasiHolmgren2} 
to prove Holgrem's Uniqueness Theorem.
\begin{Rem}
We have formulated our results for scalar operators on open sets of $\R^n$ but they remain of course valid on
ultradifferentiable manifolds of class $\{\M\}$. Actually, all the conclusions in this section hold even for
determined systems of operators and vector-valued distributions.
Indeed, we have only to verify that Theorem \ref{sec:Uniq-thm1} holds also for distributions with values in $\C^\nu$, but this is trivial:
If $f(x)\leq f(x_0)$ for $x\in\supp u$ then $f(x)\leq f(x_0)$ for all $x\in \supp u_j$ and any 
$1\leq j\leq n$, since $\supp u=\bigcup_{j=1}^\nu\supp u_j$. 
Hence Theorem \ref{sec:Uniq-thm1} implies
\begin{equation*}
(x_0,\pm df(x_0))\in\bigcap_{j=1}^\nu \WF_\M u_j\subseteq\WF_\M u.
\end{equation*}
\end{Rem}
Following an idea of Bony (\cite{MR0241805,MR0474426}) it is possible to generalize the results above.
For the formulation we need some additional notation. 
Consider a smooth real valued function $p$ on $T^\ast \Omega$. The \emph{Hamiltonian vector field}
 $H_p$ of $p$ is defined by 
\begin{equation*}
H_p=\sum_{j=1}^n\biggl(\frac{\partial p}{\partial \xi_j}\frac{\partial}{\partial x_j}-\frac{\partial p}{\partial x_j}
\frac{\partial}{\partial \xi_j}\biggr).
\end{equation*}
An integral curve of $H_p$, i.e.\ a solution of the Hamilton-Jacobi equations
\begin{align*}
\frac{d x_j}{dt}&=\frac{\partial p}{\partial \xi_j}(x,\xi),\\
\frac{d\xi_j}{dt}&=-\frac{\partial p}{\partial x_j}(x,\xi),
\end{align*}
$j=1,\dotsc,n$, is called a \emph{bicharacteristic} if $p$ vanishes on it. 
If $q$ is another smooth real valued function on $T^\ast\Omega$ then the \emph{Poisson bracket} is defined by 
$\{p,q\}:=H_p(q)$ or in coordinates
\begin{equation*}
\{p,q\}=\sum_{j=1}^n\biggl(\frac{\partial p}{\partial \xi_j}\frac{\partial q}{\partial x_j}
-\frac{\partial p}{\partial x_j}\frac{\partial q}{\partial \xi_j}\biggr).
\end{equation*}
See \cite{MR1269107} or \cite{MR1996773} for more details.

We continue by recalling a result of Sj\"ostrand \cite{MR699623} (see also \cite{MR1996773}).
\begin{Thm}\label{sec:Uniq-thm3}
Let $F$ be a closed subset of $\Omega$ and suppose that $p\in\E(T^\ast\Omega\setminus\{0\})$ 
is real valued and vanishes on $N_e(F)$.
If $(x_0,\xi_0)\in N_e(F)$ then the bicharacteristic $t\mapsto (x(t),\xi(t))$ with $(x(0),\xi(0))=(x_0,\xi_0)$ 
stays for $\lvert t\rvert$ small in $N_e(F)$.
\end{Thm}
The analogous statement is of course also true for $N_i(F)$ replacing $N_e(F)$. It follows
\begin{Cor}[Bony]
Let $F$ be a closed subset of $\Omega$ and set
\begin{equation*}
\mathcal{N}_F:=\bigl\{p\in\E(T^\ast\Omega\!\setminus\!\{0\})\mid p\equiv 0\text{ on } N (F)\bigr\}.
\end{equation*}
Then $\mathcal{N}_F$ is an ideal in $\E(T^\ast\Omega\!\setminus\!\{0\})$ that is closed under Poisson brackets.
\end{Cor}
We obtain the quasianalytic version of a result of Bony \cite{MR0241805,MR0474426}.
\begin{Thm}
Let $P$ a differential operator with $\E_\M$-coefficients on $\Omega$ 
and $\Pi$ the Poisson algebra that is generated by all functions $f\in\E(T^\ast\Omega\!\setminus\!\{0\})$ 
that vanish on $\Char P$.

If $u\in\D^\prime(\Omega)$ is a solution of the homogeneous equation $Pu=0$ then all functions in $\Pi$ 
have to vanish on $N(\supp u)$.
\end{Thm}
\begin{Cor}
If the elements of $\Pi$ have no common zeros and $u$ vanishes in a neighbourhood of a point 
$p_0\in\Omega$ then $u$ must vanish in the connected component of $\Omega$ that contains $p_0$.
\end{Cor}

We continue by taking a closer look at Theorem \ref{sec:Uniq-thm3}. 
Let $\pi: T^\ast\Omega \rightarrow \Omega$ be the canonical projection and 
$(x_0,\xi_0)\in T^\ast\Omega\!\setminus\!\{0\}$.
If $q$ is a smooth function on  $T^\ast\Omega\!\setminus\!\{0\}$ that vanishes on $N(F)$, 
$F\subseteq\Omega$ closed, and $\lambda(t)$ the bicharacteristic through $(x_0,\xi_0)$
then we conclude that the bicharacteristic curve $\gamma(t)= \pi\circ\lambda$ must stay in $\partial F$ for
small $t$ in view of  the remarks before Theorem \ref{sec:Uniq-thm2}.

Now suppose that $Q$ is a real vector field on $\Omega$ and $q$ its symbol.
If we denote by $\gamma$ the integral curve of $Q$ through $x_0$ and by $\lambda$ the bicharacteristic
of $q$ through $(x_0,\xi_0)$ where $(x_0,\xi_0)$ then it is trivial that $\gamma=\pi\circ \lambda$.

\begin{Def}\label{sec:Uniq-def}
We say that a partial differential operator $P$ on $\Omega$ with $\E_\M$-coefficients is $\M$-admissible iff
there are ultradifferentiable real-valued vector fields $Q_1,\dotsc,Q_d$ with symbols $q_1,\dotsc,q_d$ such that
each $q_j$ vanishes on $\Char P$.
\end{Def}
Following the approach of Sj\"ostrand \cite{MR699623} we can generalize results of Zachmanoglou 
\cite{MR0299925} (c.f.\ also Bony \cite{MR0474426}) to the quasianalytic setting.
\begin{Prop}\label{sec:Uniq-prop2}
Let $P$ be an $\M$-admissible operator. 
If $\mathcal{L}=\mathcal{L}(Q_1,\dotsc,Q_d)$ is the Lie algebra generated by the 
vector fields $Q_j$, $j=1,\dotsc,d$, 
$\varphi\in\CC^1(\Omega,\R)$ near a point $x_0\in\Omega$ such that $(x_0,\varphi^\prime(x_0))\in\Char P$
and $u\in\D^\prime(\Omega)$ a solution of $Pu=0$ such that near 
$x_0$ we have $x_0\in\supp u\subseteq\{\varphi\geq 0\}$.
Then each $Q\in\mathcal{L}$ is tangent to $\{\varphi=0\}$ at $x_0$ 
and the local Nagano leaf $\gamma_{x_0}(\mathcal{L})$ is contained in $\supp u$.
\end{Prop}
\begin{proof}
By assumption all $Q_1,\dotsc,Q_d$ are tangent to $\{\varphi=0\}$ at $x_0$ and hence also all 
$Q\in\mathcal{L}$. From the remarks before Definition  \ref{sec:Uniq-def} and Theorem \ref{sec:Uniq-thm2}
 we see that all integral curves of the vector fields in $\mathcal{L}$ must be contained in $\partial\supp u$ 
 for a small neighbourhood of $x_0$. 
 Inspecting the construction of the representative of the local Nagano leaf in the proof of Theorem \ref{NaganoThm} we 
 see that $\gamma_{x_0}(\mathcal{L})\subseteq\supp u$ near $x_0$.
\end{proof}
In fact, we have the following global theorem (see for the analytic case Zachmanoglou \cite{MR0299925}, 
c.f.\ Bony \cite[Theorem 2.4.]{MR0474426})
\begin{Thm}\label{GlobalUniq}
Let $P$ an $\M$-admissable differential operator. 
If $u\in\D^\prime (\Omega)$ is a solution of $Pu=0$ and $p_0\notin\supp u$ then every integral curve of
the vector fields $Q_1,\dotsc, Q_d$ through $p_0$ stays in $\Omega\setminus\supp u$.
\end{Thm}
\begin{proof}
Let $\Gamma=\Gamma_{p_0}(\mathcal{L})$ be the global Nagano leaf of $\mathcal{L}=\mathcal{L}(Q_1,\dotsc,Q_d)$ through $p_0$ 
and suppose that $\partial \supp u\cap \Gamma\neq\emptyset$.
Then there has to be a point $q_0\in\Gamma\cap\partial\supp u$ such that 
for all neighbourhoods $V\subseteq\Omega$ of $x_0$ we have
\begin{equation*}
\bigl(\Gamma\cap V\bigr)\cap \bigl(\Omega\setminus\supp u\bigr)\neq\emptyset.
\end{equation*}
Let $V$ be small enough such that $\Gamma\cap V$ is the representative of the local Nagano leaf of $\mathcal{L}$ 
at $q_0$
constructed in the proof of Theorem \ref{NaganoThm}. Then
\begin{equation*}
\Gamma\setminus\supp u\cap V\neq \emptyset.
\end{equation*}
Thence there is a vector field $X\in \mathcal{L}$ such that if $\gamma(t)=\exp tX$ is the integral curve 
of $X$ through $q_0$  then $\gamma(0)=q_0$ and $\gamma(1)=q_1\in V\!\setminus\!\supp u$.
Possibly shrinking $V$ and applying an ultradifferentiable coordinate change in $V$ we may assume that
 $q_0=0$, $q_1=(0,\dotsc,0,1)$ and 
\begin{equation*}
X=\frac{\partial}{\partial x_n}.
\end{equation*}
We note that in these new coordinates the assumption on $P$ can be stated in the following way.
Let $\xi\in\R^n$ with $\xi_n\neq 0$ then $p_m(x,\xi)\neq 0$ for all $x\in V$.
There is also a neighbourhood $V_1\subseteq V$ of $q_1$ such that $u$ vanishes on $V_1$.

We adapt the proof of Zachmanoglou \cite[Theorem 1]{MR0240442}. Let $r>0$ and $\delta>0$ be small enough so that
\begin{equation*}
U=\bigl\{x\in\R^n\mid \lvert x^\prime\rvert <r,\;-\delta<x_n<1\bigr\}
\end{equation*}
is contained in $V$ and 
\begin{equation*}
\bigl\{x\in\R^n\mid \lvert x^\prime\rvert<r,\;x_n=1\bigr\}\subseteq V_1.
\end{equation*}
We consider the real-analytic function
\begin{equation*}
F(x)=(1+\delta)\frac{\vert x^\prime\rvert^2}{r^2}- \delta -x_n.
\end{equation*}
The normals of the level hypersurfaces of $F$ are always nonzero in the direction of the $n$-th unit vector.
It follows that the level hypersurfaces are everywhere non-characteristic with respect to $P$ in $V$.
Set
\begin{equation*}
U_1=\biggl\{x\in U:\; F(x) < -\frac{\delta}{2}\biggr\}
\end{equation*}
and note that if $x\in U_1$ then $x_n>-\delta /2$. 
It is easy to see that
$U_1\cap\supp u$ is relatively compact in $U$. 
We conclude that $u=0$ in $U_1$ by Corollary \ref{QuasiHolmgren2}.
That is a contradiction to the assumption $q_0\in\partial\supp u$.
\end{proof}
If $Q_1,\dotsc,Q_d$ are real valued vector fields with $\E_\M$-coefficients, then the operators
\begin{align*}
P_0&=Q_1+iQ_2\\
P_k&=\sum_{j=1}^d Q_j^{2k}\qquad k\in\N
\end{align*}
are $\M$-admissible.

For our last result we need to recall the notion of finite type which was introduced by H{\"o}rmander \cite{Hoermander1967}. 
We say a collection of smooth real vector fields $X_1,\dotsc , X_d$ on $\Omega$ 
is of finite type (of length at most $r$)  if at any point $p\in\Omega $ the tangent space $T_p\Omega$
is generated by $X_j(p)$ and some iterated commutators 
$[X_{i_1},[X_{i_2},[\dotsc,[X_{i_{q-1}},X_{i_q}]\dotsc]]](p)$, where $q\leq r$.

A straightforward application of Theorem \ref{GlobalUniq} gives the following corollary.
\begin{Cor}
Let $\Omega$ be connected and assume that the real vector fields 
$X_1,\dotsc,X_d$ are of class $\{\M\}$ and of finite type and 
let $u\in\D^\prime(\Omega)$ be a solution of $P_ku=0$. If $u$ vanishes on an open subset of $\Omega$ then $u\equiv 0$ in $\Omega$.
\end{Cor}
\bibliographystyle{plain}
\bibliography{UltraRef1}

\begin{thebibliography}{10}

\bibitem{MR2718658}
Ziad Adwan and Gustavo Hoepfner.
\newblock Approximate solutions and micro-regularity in the {D}enjoy-{C}arleman
  classes.
\newblock {\em J. Differential Equations}, 249(9):2269--2286, 2010.

\bibitem{MR2595651}
A.~A. Albanese, D.~Jornet, and A.~Oliaro.
\newblock Quasianalytic wave front sets for solutions of linear partial
  differential operators.
\newblock {\em Integral Equations Operator Theory}, 66(2):153--181, 2010.

\bibitem{MR1668103}
M.~Salah Baouendi, Peter Ebenfelt, and Linda~Preiss Rothschild.
\newblock {\em Real submanifolds in complex space and their mappings},
  volume~47 of {\em Princeton Mathematical Series}.
\newblock Princeton University Press, Princeton, NJ, 1999.

\bibitem{MR3701268}
S.~Berhanu and Abraham Hailu.
\newblock Characterization of {G}evrey regularity by a class of {FBI}
  transforms.
\newblock In {\em Recent applications of harmonic analysis to function spaces,
  differential equations, and data science}, Appl. Numer. Harmon. Anal., pages
  451--482. Birkh\"auser/Springer, Cham, 2017.

\bibitem{MR2864805}
S.~Berhanu and J.~Hounie.
\newblock A class of {FBI} transforms.
\newblock {\em Comm. Partial Differential Equations}, 37(1):38--57, 2012.

\bibitem{MR2397326}
Shiferaw Berhanu, Paulo~D. Cordaro, and Jorge Hounie.
\newblock {\em An introduction to involutive structures}, volume~6 of {\em New
  Mathematical Monographs}.
\newblock Cambridge University Press, Cambridge, 2008.

\bibitem{MR2061220}
Edward Bierstone and Pierre~D. Milman.
\newblock Resolution of singularities in {D}enjoy-{C}arleman classes.
\newblock {\em Selecta Math. (N.S.)}, 10(1):1--28, 2004.

\bibitem{MR0203201}
G{\"o}ran Bj{\"o}rck.
\newblock Linear partial differential operators and generalized distributions.
\newblock {\em Ark. Mat.}, 6:351--407 (1966), 1966.

\bibitem{MR1382568}
Jan Boman.
\newblock Microlocal quasianalyticity for distributions and ultradistributions.
\newblock {\em Publ. Res. Inst. Math. Sci.}, 31(6):1079--1095, 1995.

\bibitem{MR2387040}
Jos{\'e} Bonet, Reinhold Meise, and Sergej~N. Melikhov.
\newblock A comparison of two different ways to define classes of
  ultradifferentiable functions.
\newblock {\em Bull. Belg. Math. Soc. Simon Stevin}, 14(3):425--444, 2007.

\bibitem{MR0241805}
J.~M. Bony.
\newblock Une extension du th{\'e}or{\`e}me de {H}olmgren sur l'unicit{\'e} du
  probl{\`e}me de {C}auchy.
\newblock {\em C. R. Acad. Sci. Paris S{\'e}r. A-B}, 268:A1103--A1106, 1969.

\bibitem{MR0474426}
J.~M. Bony.
\newblock Extensions du th{\'e}or{\`e}me de {H}olmgren.
\newblock In {\em S{\'e}minaire {G}oulaouic-{S}chwartz (1975/1976)
  \'{E}quations aux d{\'e}riv{\'e}es partielles et analyse fonctionnelle,
  {{E}xp. {N}o. 17}}, page~13. Centre Math., {\'E}cole Polytech., Palaiseau,
  1976.

\bibitem{MR0650834}
J.~M. Bony.
\newblock \'{E}quivalence des diverses notions de spectre singulier analytique.
\newblock In {\em S{\'e}minaire {G}oulaouic-{S}chwartz (1976/1977),
  \'{E}quations aux d{\'e}riv{\'e}es partielles et analyse fonctionnelle,
  {E}xp. {N}o. 3}, page~12. Centre Math., {\'E}cole Polytech., Palaiseau, 1977.

\bibitem{MR1052587}
R.~W. Braun, R.~Meise, and B.~A. Taylor.
\newblock Ultradifferentiable functions and {F}ourier analysis.
\newblock {\em Results Math.}, 17(3-4):206--237, 1990.

\bibitem{MR0399494}
J.~Bros and D.~Iagolnitzer.
\newblock Tubo\"\i des et structure analytique des distributions. {II}.
  {S}upport essentiel et structure analytique des distributions.
\newblock In {\em S{\'e}minaire {G}oulaouic-{L}ions-{S}chwartz 1974--1975:
  \'{E}quations aux d{\'e}riv{\'e}es partielles lin{\'e}aires et non
  lin{\'e}aires, {E}xp. {N}o. 18}, page~34. Centre Math., {\'E}cole Polytech.,
  Paris, 1975.

\bibitem{zbMATH02598188}
T.~{Carleman}.
\newblock {Sur les fonctions ind\'efiniment d\'erivables.}
\newblock {\em {C. R. Acad. Sci., Paris}}, 177:422--424, 1923.

\bibitem{zbMATH02599917}
T.~{Carleman}.
\newblock {Sur les fonctions quasi-analytiques.}
\newblock {5. Kongre{\ss} der Skandinav. Mathematiker in Helsingfors, 4.--7.
  Juli 1922. Helsingfors: Akadem. Buchh., 181-196 (1923).}, 1923.

\bibitem{MR678605}
Jacques Chazarain and Alain Piriou.
\newblock {\em Introduction to the theory of linear partial differential
  equations}, volume~14 of {\em Studies in Mathematics and its Applications}.
\newblock North-Holland Publishing Co., Amsterdam-New York, 1982.
\newblock Translated from the French.

\bibitem{MR1492944}
Soon-Yeong Chung and Dohan Kim.
\newblock A quasianalytic singular spectrum with respect to the
  {D}enjoy-{C}arleman class.
\newblock {\em Nagoya Math. J.}, 148:137--149, 1997.

\bibitem{zbMATH02601219}
A.~{Denjoy}.
\newblock {Sur les fonctions quasi-analytiques de variable r\'eelle.}
\newblock {\em {C. R. Acad. Sci., Paris}}, 173:1329--1331, 1921.

\bibitem{MR2680692}
J.~J. Duistermaat and J.~A.~C. Kolk.
\newblock {\em Distributions}.
\newblock Cornerstones. Birkh{\"a}user Boston, Inc., Boston, MA, 2010.
\newblock Theory and applications, Translated from the Dutch by J. P. van Braam
  Houckgeest.

\bibitem{MR0587795}
E.~M. Dyn'kin.
\newblock Pseudoanalytic continuation of smooth functions. {U}niform scale.
\newblock In {\em Mathematical programming and related questions ({P}roc.
  {S}eventh {W}inter {S}chool, {D}rogobych, 1974), {T}heory of functions and
  functional analysis ({R}ussian)}, pages 40--73. Central {\`E}konom.-Mat.
  Inst. Akad. Nauk SSSR, Moscow, 1976.

\bibitem{zbMATH03751341}
E.~M. Dyn'kin.
\newblock {Pseudoanalytic extension of smooth functions. The uniform scale.}
\newblock {\em {Transl., Ser. 2, Am. Math. Soc.}}, 115:33--58, 1980.

\bibitem{MR1253229}
E.~M. Dyn'kin.
\newblock The pseudoanalytic extension.
\newblock {\em J. Anal. Math.}, 60:45--70, 1993.

\bibitem{MR0285849}
Leon Ehrenpreis.
\newblock {\em Fourier analysis in several complex variables}.
\newblock Pure and Applied Mathematics, Vol. XVII. Wiley-Interscience
  Publishers A Division of John Wiley \& Sons, New York-London-Sydney, 1970.

\bibitem{WeightMatrix}
S.~F{\"u}rd{\"o}s, D.~Nenning, A.~Rainer, and G.~Schindl.
\newblock In preparation, 2019.

\bibitem{Furdos2}
Stefan F{\"u}rd{\"o}s.
\newblock Ultradifferentiable {CR} manifolds.
\newblock Preprint, 2018.

\bibitem{MR1269107}
Alain Grigis and Johannes Sj{\"o}strand.
\newblock {\em Microlocal analysis for differential operators}, volume 196 of
  {\em London Mathematical Society Lecture Note Series}.
\newblock Cambridge University Press, Cambridge, 1994.
\newblock An introduction.

\bibitem{Hoepfner-MedradoPreprint}
G.~Hoepfner and A.~Medrado.
\newblock The {FBI} transforms and their use in microlocal analysis.
\newblock Preprint, 2017.

\bibitem{zbMATH02662678}
E.~{Holmgren}.
\newblock {\"Uber Systeme von linearen partiellen Differentialgleichungen.}
\newblock {\em {{\"O}fversigt af Kongl. Vetenskaps-Akad. F{\"o}r.}},
  58:91--103, 1901.

\bibitem{Hoermander1967}
L.~{H\"ormander}.
\newblock Hypoelliptic second order differential equations.
\newblock {\em Acta Mathematica}, 119:147--171, 1967.

\bibitem{MR0513000}
Lars H{\"o}rmander.
\newblock Linear differential operators.
\newblock In {\em Actes du {C}ongr{\`e}s {I}nternational des
  {M}ath{\'e}maticiens ({N}ice, 1970), {T}ome 1}, pages 121--133.
  Gauthier-Villars, Paris, 1971.

\bibitem{MR0294849}
Lars H{\"o}rmander.
\newblock Uniqueness theorems and wave front sets for solutions of linear
  differential equations with analytic coefficients.
\newblock {\em Comm. Pure Appl. Math.}, 24:671--704, 1971.

\bibitem{MR0404822}
Lars H{\"o}rmander.
\newblock {\em Linear partial differential operators}.
\newblock Springer Verlag, Berlin-New York, 1976.

\bibitem{MR1275197}
Lars H{\"o}rmander.
\newblock Remarks on {H}olmgren's uniqueness theorem.
\newblock {\em Ann. Inst. Fourier (Grenoble)}, 43(5):1223--1251, 1993.

\bibitem{MR1996773}
Lars H{\"o}rmander.
\newblock {\em The analysis of linear partial differential operators. {I}}.
\newblock Classics in Mathematics. Springer-Verlag, Berlin, 2003.
\newblock Distribution theory and Fourier analysis, Reprint of the second
  (1990) edition.

\bibitem{MR1826922}
June~Gi Kim, Soon-Yeong Chung, and Dohan Kim.
\newblock Microlocal analysis in the {D}enjoy-{C}arleman class.
\newblock {\em J. Korean Math. Soc.}, 38(3):561--575, 2001.

\bibitem{MR0320743}
Hikosaburo Komatsu.
\newblock Ultradistributions. {I}. {S}tructure theorems and a characterization.
\newblock {\em J. Fac. Sci. Univ. Tokyo Sect. IA Math.}, 20:25--105, 1973.

\bibitem{MR531445}
Hikosaburo Komatsu.
\newblock The implicit function theorem for ultradifferentiable mappings.
\newblock {\em Proc. Japan Acad. Ser. A Math. Sci.}, 55(3):69--72, 1979.

\bibitem{MR575993}
Hikosaburo Komatsu.
\newblock Ultradifferentiability of solutions of ordinary differential
  equations.
\newblock {\em Proc. Japan Acad. Ser. A Math. Sci.}, 56(4):137--142, 1980.

\bibitem{MR1178558}
Hikosaburo Komatsu.
\newblock Microlocal analysis in {G}evrey classes and in complex domains.
\newblock In {\em Microlocal analysis and applications ({M}ontecatini {T}erme,
  1989)}, volume 1495 of {\em Lecture Notes in Math.}, pages 161--236.
  Springer, Berlin, 1991.

\bibitem{Lerner2010}
Nicolas {Lerner}.
\newblock {\em Metrics on the phase space and non-selfadjoint
  pseudo-differential operators.}
\newblock Basel: Birkh\"auser, 2010.

\bibitem{MR1806500}
Otto Liess.
\newblock Carleman regularization in the {${C}^\infty$}-category.
\newblock {\em Ann. Univ. Ferrara Sez. VII (N.S.)}, 45(suppl.):213--240 (2000),
  1999.
\newblock Workshop on Partial Differential Equations (Ferrara, 1999).

\bibitem{MR0051893}
S.~Mandelbrojt.
\newblock {\em S{\'e}ries adh{\'e}rentes, r{\'e}gularisation des suites,
  applications}.
\newblock Gauthier-Villars, Paris, 1952.

\bibitem{MR0412461}
J.~N. Mather.
\newblock On {N}irenberg's proof of {M}algrange's preparation theorem.
\newblock In {\em Proceedings of {L}iverpool {S}ingularities---{S}ymposium, {I}
  (1969/70)}, pages 116--120. Lecture Notes in Mathematics, Vol. 192. Springer,
  Berlin, 1971.

\bibitem{MR0431289}
Anders Melin and Johannes Sj{\"o}strand.
\newblock Fourier integral operators with complex-valued phase functions.
\newblock In {\em Fourier integral operators and partial differential equations
  ({C}olloq. {I}nternat., {U}niv. {N}ice, {N}ice, 1974)}, pages 120--223.
  Lecture Notes in Math., Vol. 459. Springer, Berlin, 1975.

\bibitem{MR0199865}
Tadashi Nagano.
\newblock Linear differential systems with singularities and an application to
  transitive {L}ie algebras.
\newblock {\em J. Math. Soc. Japan}, 18:398--404, 1966.

\bibitem{MR737333}
Hans-Joachim Petzsche and Dietmar Vogt.
\newblock Almost analytic extension of ultradifferentiable functions and the
  boundary values of holomorphic functions.
\newblock {\em Math. Ann.}, 267(1):17--35, 1984.

\bibitem{MR3463541}
Stevan Pilipovi{\'c}, Nenad Teofanov, and Filip Tomi{\'c}.
\newblock Beyond {G}evrey regularity.
\newblock {\em J. Pseudo-Differ. Oper. Appl.}, 7(1):113--140, 2016.

\bibitem{Pilipovic:2016oj}
Stevan Pilipovi{\'c}, Nenad Teofanov, and Filip Tomi{\'c}.
\newblock Superposition and propagation of singularities for extended gevrey
  regularity.
\newblock 07 2016.

\bibitem{MR3285413}
Armin Rainer and Gerhard Schindl.
\newblock Composition in ultradifferentiable classes.
\newblock {\em Studia Math.}, 224(2):97--131, 2014.

\bibitem{MR3462072}
Armin Rainer and Gerhard Schindl.
\newblock Equivalence of stability properties for ultradifferentiable function
  classes.
\newblock {\em Rev. R. Acad. Cienc. Exactas F\'\i s. Nat. Ser. A Math. RACSAM},
  110(1):17--32, 2016.

\bibitem{MR1249275}
Luigi Rodino.
\newblock {\em Linear partial differential operators in {G}evrey spaces}.
\newblock World Scientific Publishing Co., Inc., River Edge, NJ, 1993.

\bibitem{MR0158261}
Charles Roumieu.
\newblock Ultra-distributions d{\'e}finies sur {$R^{n}$} et sur certaines
  classes de vari{\'e}t{\'e}s diff{\'e}rentiables.
\newblock {\em J. Analyse Math.}, 10:153--192, 1962/1963.

\bibitem{MR924157}
Walter Rudin.
\newblock {\em Real and complex analysis}.
\newblock McGraw-Hill Book Co., New York, third edition, 1987.

\bibitem{zbMATH03331746}
M.~{Sato}.
\newblock {Hyperfunctions and partial differential equations.}
\newblock In {\em {Proc. Int. Conf. Funct. Anal. Rel. Topics, Tokyo 1969.}},
  pages 91--94. Tokyo University Press, 1970.

\bibitem{MR0420735}
Mikio Sato, Takahiro Kawai, and Masaki Kashiwara.
\newblock Microfunctions and pseudo-differential equations.
\newblock In {\em Hyperfunctions and pseudo-differential equations ({P}roc.
  {C}onf., {K}atata, 1971; dedicated to the memory of {A}ndr{\'e}
  {M}artineau)}, pages 265--529. Lecture Notes in Math., Vol. 287. Springer,
  Berlin, 1973.

\bibitem{MR699623}
Johannes Sj{\"o}strand.
\newblock Singularit{\'e}s analytiques microlocales.
\newblock In {\em Ast{\'e}risque, 95}, volume~95 of {\em Ast{\'e}risque}, pages
  1--166. Soc. Math. France, Paris, 1982.

\bibitem{MR2011916}
Vincent Thilliez.
\newblock Division by flat ultradifferentiable functions and sectorial
  extensions.
\newblock {\em Results Math.}, 44(1-2):169--188, 2003.

\bibitem{MR2384272}
Vincent Thilliez.
\newblock On quasianalytic local rings.
\newblock {\em Expo. Math.}, 26(1):1--23, 2008.

\bibitem{MR1128962}
Takesi Yamanaka.
\newblock On {ODE}s in the ultradifferentiable class.
\newblock {\em Nonlinear Anal.}, 17(7):599--611, 1991.

\bibitem{MR0240442}
E.~C. Zachmanoglou.
\newblock An application of {H}olmgren's theorem and convexity with respect to
  differential operators with flat characteristic cones.
\newblock {\em Trans. Amer. Math. Soc.}, 140:109--115, 1969.

\bibitem{MR0299925}
E.~C. Zachmanoglou.
\newblock Foliations and the propagation of zeroes of solutions of partial
  differential equations.
\newblock {\em Bull. Amer. Math. Soc.}, 78:835--839, 1972.

\end{thebibliography}
\end{document}